\documentclass[reqno]{amsart}

\usepackage[numbers]{natbib}
\usepackage{amssymb}
\usepackage{float}
\usepackage{caption}
\usepackage{footnote}
\usepackage{tikz, tikz-cd}
\usetikzlibrary{matrix, arrows}
\usetikzlibrary{patterns}
\usepackage{wrapfig}
\usepackage{hyperref}
\usepackage{array}
\hypersetup{colorlinks,linkcolor={blue},citecolor={blue},urlcolor={blue}}
\usepackage[cal=boondox]{mathalfa}
\usepackage{graphicx}
\usepackage[abs]{overpic}
\graphicspath{ {images/} }

\newtheorem{theorem}{Theorem}[section]

\newtheorem{lemma}[theorem]{Lemma}
\newtheorem{proposition}[theorem]{Proposition}
\newtheorem{corollary}[theorem]{Corollary}

\theoremstyle{definition}
\newtheorem{definition}[theorem]{Definition}
\newtheorem{notation}[theorem]{Notation}

\numberwithin{equation}{section}

\newcommand{\bR}{\mathbb{R}}
\newcommand{\bQ}{\mathbb{Q}}
\newcommand{\bZ}{\mathbb{Z}}
\newcommand{\bN}{\mathbb{N}}
\newcommand{\bC}{\mathbb{C}}
\newcommand{\bH}{\mathbb{H}}
\newcommand{\cO}{\mathcal{O}}

\newcommand{\cj}{\mathcal{j}}
\newcommand{\cJ}{\mathcal{J}}
\newcommand{\cB}{\mathcal{B}}
\newcommand{\cS}{\mathcal{S}}

\newcommand{\fp}{\mathfrak{p}}
\newcommand{\fq}{\mathfrak{q}}
\newcommand{\fa}{\mathfrak{a}}

\newcommand{\diam}{\raisebox{0.6pt}[0pt][0pt]{$\Diamond$}}

\begin{document}

\title[Bianchi group generators]{Bounds on entries in Bianchi group generators}

\thanks{The author is grateful to John Cremona and Alexander Rahm for their time, expertise, and data, which were produced with their code. The author also thanks the anonymous referee for corrections and very helpful suggestions.}

\author{Daniel E. Martin}
\address{University of California, Davis, CA, United States}
\email{dmartin@math.ucdavis.edu}



\date{\today}

\begin{abstract}Upper and lower bounds are given for the maximum Euclidean curvature among faces in Bianchi's fundamental polyhedron for $\text{PSL}_2(\cO)$ in the upper-half space model of hyperbolic space, where $\cO$ is an imaginary quadratic ring of integers with discriminant $\Delta$. We prove these bounds are asymptotically within $(\log |\Delta|)^{8.54}$ of one another. This improves on the previous best upper-bound, which is roughly off by a factor between $\Delta^2$ and $|\Delta|^{5/2}$ depending on the smallest prime dividing $\Delta$. The gap between our upper and lower bounds is determined by an analog of Jacobsthal's function, introduced here for imaginary quadratic fields.\end{abstract}

\maketitle

\section{Introduction}\label{sec:1}

Bianchi groups are of the form $\text{PSL}_2(\cO)$, where $\cO$ is an imaginary quadratic ring of integers. They are exceptional in that $\text{PSL}_n(\cO)$ is generated by elementary matrices whenever $n>2$ \cite{bass} or $\cO$ is any ring of integers other than non-Euclidean, imaginary quadratic \cite{cohn} \cite{vaser}. But for the Bianchi groups, no general, explicit description of a generating set is known. Instead presentations are computed algorithmically.

The theory behind many such algorithms \cite{aranes} \cite{cremona} \cite{floge} \cite{rahm} \cite{riley} \cite{swan} \cite{vogtmann} \cite{yasaki} is rooted in the action of $\text{PSL}_2(\cO)$ on hyperbolic 3-space. We use the upper-half space model, $\bH=\{(\zeta,t)\,|\,\zeta\in\bC,t\in(0,\infty)\}$, equipped with the hyperbolic metric. The Bianchi group action on $\bH$ admits a fundamental domain---a closed subset whose orbit tessellates $\bH$. (See Sections 1.1, 2.2, and 7.3 of \cite{elstrodt} for details and background.) There is one such domain in particular that interests us.

\begin{definition}\label{def:bianchi}Let $\cO$ be the integers in an imaginary quadratic field $K$ of discriminant $\Delta$, and let $\Gamma=\text{PSL}_2(\cO)$. Given a fundamental parallelogram $F\subset\bC$ for $\cO$, the \emph{Bianchi polyhedron} (also called a \emph{Ford domain}) is $$\cB=\{P\in\bH\,|\,\zeta(P)\in \overline{F},\,t(P)\geq t(gP)\text{ for any }g\in\Gamma\},$$ where $\zeta(P)$ and $t(P)$ denote the first and second coordinates of $P$.\end{definition}

Bianchi first computed these polyhedra for 11 different discriminants \cite{bianchi}. 

The second coordinate from the action of $g\in\Gamma$ on $P\in \bH$ is \begin{equation}\label{eq:1}t(gP)=\frac{t(P)}{|\mu\zeta(P)-\lambda|^2+|\mu|^2t(P)^2},\hspace{0.5cm}\text{where}\hspace{0.5cm}g=\begin{bmatrix}\beta & -\alpha\\ -\mu & \lambda\end{bmatrix}.\end{equation} 
From (\ref{eq:1}) we see that $t(gP)=t(P)$ exactly when $|\mu\zeta(P)-\lambda|^2+|\mu|^2t(P)^2=1$. This equation holds for all $P$ if $\mu=0$ since $|\lambda|=1$ is forced by $(\lambda,\mu)=\cO$. If $\mu\neq 0$, it defines the following surface.

\begin{notation}For coprime $\lambda,\mu\in\cO$ with $\mu\neq 0$, let $S_{\lambda/\mu}\subset\bH$ denote the open Euclidean hemisphere (or hyperbolic plane) with center $(\lambda/\mu,0)\in\partial\hspace{0.03cm}\bH$ and radius $1/|\mu|$. Given $g\in\Gamma$ as in (\ref{eq:1}), $S_{\lambda/\mu}$ is also written $S_g$.\end{notation}

\begin{wrapfigure}{r}{0.41\textwidth}
\vspace{-0.29cm}
    \hspace{0.24cm}\includegraphics[trim=2cm 0cm 2cm 0cm,clip,height=5.8cm]{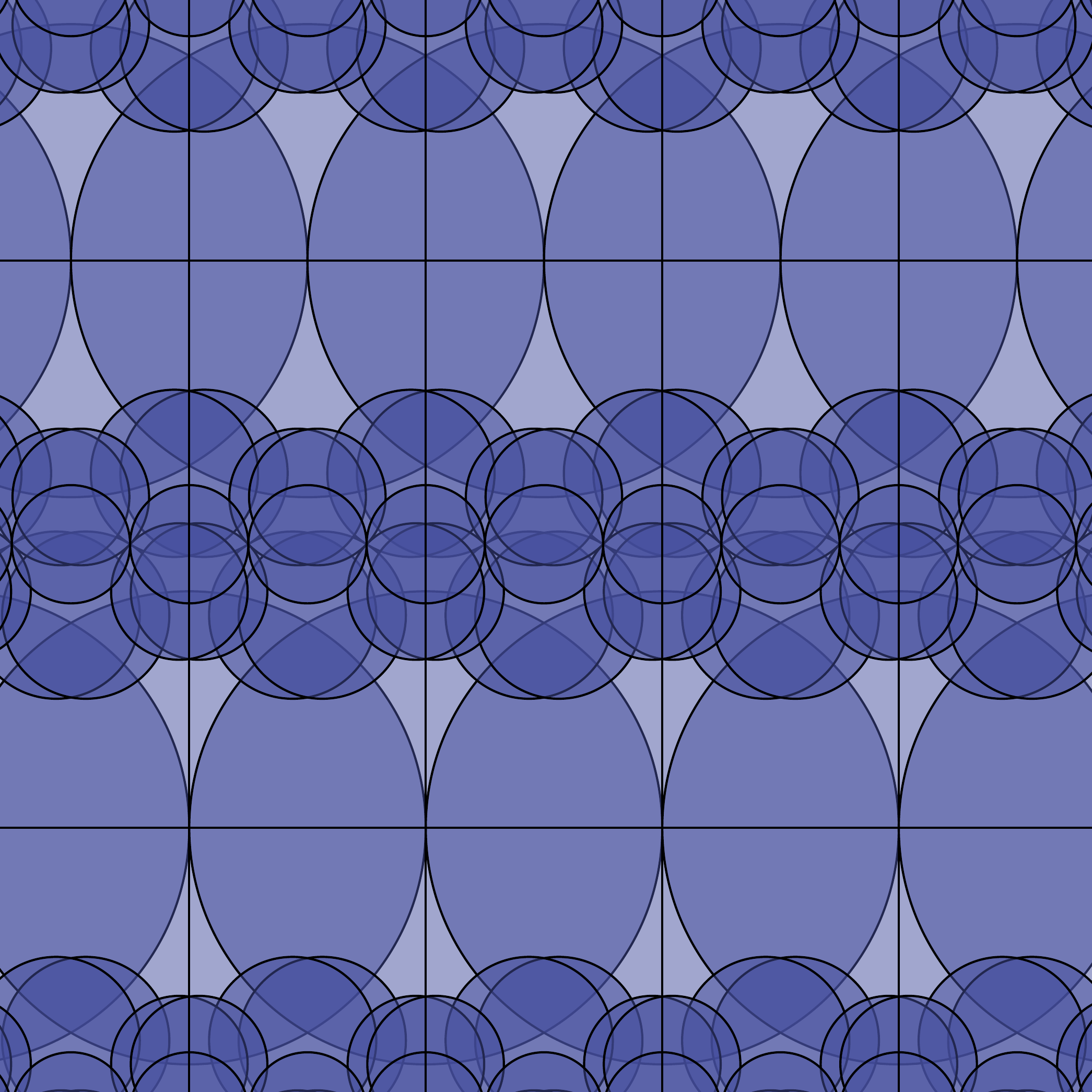}
    \captionsetup{width=0.375\textwidth}
    \setlength\abovecaptionskip{0.2cm}
    \setlength\belowcaptionskip{-0.5cm}
    \caption{Projection of $\cB$ for $\bQ(\sqrt{-23})$ tiled over $\bC$.}\label{fig:1}
\end{wrapfigure}

From (\ref{eq:1}), a point $P$ is below the hemisphere $S_g$ if and only if $t(gP)>t(P)$. So modulo the action of upper-triangular matrices, which translate $\zeta$ and fix $t$, each orbit $\Gamma P$ has a unique element that lies above every $S_{g}$---the point with maximal $t$ coordinate. Thus the assertion that $\cB$ is a fundamental domain. 

For $\Delta=-23$, Figure \ref{fig:1} shows projections of hemispheres that are large enough to be seen from above among the full set $\{S_g\}_g$. These compose the ``floor" of $\cB$. The large discs are unit discs on integers, coming from $\mu=\pm 1$. Note the choice of $F$ in Definition \ref{def:bianchi} has no effect on the list of radii of faces in $\cB$. 

The largest radius reciprocal, or curvature, in Figure \ref{fig:1} is $4$. We aim to bound this maximum from above and below for general $\Gamma$.

\begin{definition}\label{def:swan}The maximum value of $|\mu|$ for which $S_{\lambda/\mu}\cap\cB$ contains an open subset of $S_{\lambda/\mu}$ for some $\lambda$ coprime to $\mu$ is called \emph{Swan's number}, denoted $\cS(\Gamma)$.\end{definition}

A good bound on $\cS(\Gamma)$ would follow from Dirichlet's box principle if hemispheres of center $\lambda/\mu$ and radius $1/|\mu|$ were permitted for any $\lambda,\mu\in\cO$. Swan's number is only made elusive by requiring $(\lambda,\mu)=\cO$.

Poincar\'{e}'s polyhedron theorem asserts that those $g\in\Gamma$ for which $S_g$ is a face of $\cB$ form a generating set in combination with upper triangular matrices \cite{poincare}. (The theorem also gives a method for computing relations among these generators.) In particular, $\cS(\Gamma)$ indirectly bounds all matrix entries needed to generate $\Gamma$: The bottom-left entry in (\ref{eq:1}) is bounded explicitly in Definition \ref{def:swan}. Then $\beta$ and $\lambda$ in (\ref{eq:1}) can be chosen with minimal magnitude from their respective congruence classes $\text{mod}\,\mu$ via left and right multiplication by appropriate upper-triangular matrices. This, in turn, bounds the top-right entry, $\alpha=(1-\beta\lambda)/\mu$. 

From the development of Poincar\'{e}'s work by Bianchi \cite{bianchi}, Humbert \cite{humbert}, and Macbeath \cite{macbeath}, Swan designed an algorithm that computes $\cB$ and from it a presentation of $\Gamma$ \cite{swan}. Implementations of Swan's algorithm use an upper bound on $\cS(\Gamma)$ to create a superset of $\cB$'s faces, then eliminate those $S_g$ that are not needed.

To the author's knowledge, the only explicit, general upper bounds for $\cS(\Gamma)$ are due to Swan and Yao. Swan never actually states his bound in \cite{swan}, saying that it is much too large for practical use. Yao asserts that Swan's bound is roughly $|\Delta|^{4.5}$, then proves in his Theorem A that $\cS(\Gamma)<C|\Delta|^{3.5}$ for some absolute constant $C$ \cite{yao}. Aran\'{e}s claims in Remark 3.2.24 of \cite{aranes} that Swan's arguments give an upper bound with leading term $|\Delta|^3/6\sqrt{3}$. In any case, we are left with nothing of practical use and little insight into the asymptotic behavior of $\cS(\Gamma)$. As such, Swan prescribes guessing $\cB$ and computing the minimal second coordinate, $t$, among vertices in the guess. Then $\cS(\Gamma) < 1/t$ since hemispheres with larger curvature are not tall enough to poke through the floor of $\cB$. 

In correspondence, Rahm points out that $1/t$ works well in cases where computing $\cB$ is currently feasible. His implementation of Swan's algorithm (Algorithm 1 in \cite{rahm}) initially guesses that $\cS(\Gamma)$ is a multiple of $\sqrt{h|\Delta|}$, where $h$ is the class number. For small discriminants, Rahm's estimate is much closer to $\cS(\Gamma)$ than the upper bound produced here. But the forthcoming theorem shows it can be an over- or under-approximation, depending on how $\Delta$ factors, by roughly $\sqrt[4]{|\Delta|}$ for large discriminants (using $\log h\sim\log\sqrt{|\Delta|}$ \cite{siegel}).

Our upper bound for $\cS(\Gamma)$ is expressed in terms of a Jacobsthal-type function. Given a modulus $n$, Jacobsthal asked what minimal interval length would guarantee an integer coprime to $n$ \cite{jacobsthal}. We define an analogous function, $\cJ$, for imaginary quadratic fields in Definition \ref{def:jacob} and Notation \ref{not:J}.

\begin{theorem}\label{thm:intro}If $\delta\in\cO$ has maximal magnitude among proper divisors of $\Delta$ and $J\in\bZ$ is such that $\cJ(2\max(|\delta|,J\!\sqrt{|\Delta|}))\leq J$, then $$\max\!\left(\frac{|\delta|}{8},\frac{\sqrt{|\Delta|}-2}{\sqrt{3}}\right)<\cS(\Gamma)<14J\max\!\left(|\delta|,J\sqrt{|\Delta|}\right).$$\end{theorem}

For example, if $|\Delta|$ is prime then $\delta=\pm\sqrt{\Delta}$, so Theorem \ref{thm:intro} can be rewritten as $(\sqrt{|\Delta|}-2)/\sqrt{3}<\cS(\Gamma)<14J^2\!\sqrt{|\Delta|}$. But if $\Delta$ is even then $\delta=\pm\Delta/2$. For large $|\Delta|$ this makes $|\delta|>J\sqrt{|\Delta|}$, so Theorem \ref{thm:intro} becomes $|\Delta|/16<\cS(\Gamma)<7J|\Delta|$. See Figure \ref{fig:6} for a graph displaying the bounds in Theorem \ref{thm:intro} for $|\Delta|<400$.

Combining Theorem \ref{thm:intro} with the asymptotic bound for $\cJ$ coming from Proposition \ref{prop:asymptotic} gives the following.

\begin{corollary}$\cS(\Gamma)\ll J'\max(|\delta|,\,J'\!\sqrt{|\Delta|})$, where $J'=(\log |\Delta|/\log\log|\Delta|)^{ 4.27}$.\end{corollary}

So we have found $\cS(\Gamma)$ to within a power of $\log|\Delta|$. It is $|\delta|$, the largest magnitude among divisors of $\Delta$ in $\cO$. Therefore, despite the form of prior bounds and estimations, $\lim_{\Delta}(\log\cS(\Gamma)/\log|\Delta|)$ does not exist. Rather, $$\liminf_{\Delta\to-\infty}\frac{\log\cS(\Gamma)}{\log|\Delta|}=\frac{1}{2}\hspace{0.5cm}\text{and}\hspace{0.5cm}\limsup_{\Delta\to-\infty}\frac{\log\cS(\Gamma)}{\log|\Delta|}=1.$$

In Section \ref{sec:2} the lower bound in Theorem \ref{thm:intro} is proved. Section \ref{sec:3} defines our analog of Jacobsthal's function and gives it an asymptotic upper bound. Section \ref{sec:4} proves the upper bound in Theorem \ref{thm:intro}.

\section{A lower bound for Swan's number}\label{sec:2}

\begin{definition}\label{def:sing}A \emph{singular point} is any $\zeta\in K$ such that if the fractional ideal $(\zeta,1)$ equals $(\alpha,\beta)$ for $\alpha,\beta\in K$, then $|\beta|\geq 1$.\end{definition}

As written, only $\zeta$ with $(\zeta,1)$ non-principal might possibly be a singular point, because $(\zeta,1) = (\alpha,0)$ violates the lower bound on $|\beta|$. Conversely, every nontrivial ideal class has corresponding singular points. Indeed, $(\zeta,1)$ has some nonzero element of minimal magnitude, say $\beta$. It is a fact of Dedekind domains that there exists $\alpha$ with $(\alpha,\beta)=(\zeta,1)$. Then if $\alpha\neq 0$, it follows from the definition that $\alpha/\beta$ is singular.

The singular points for $\bQ(\sqrt{-23})$ are seen in Figure \ref{fig:1} at the intersections of its smallest discs' boundaries. Figure \ref{fig:2} gives a close-up of the three (up to translation) singular points for $\bQ(\sqrt{-34})$, one for each nontrivial ideal class. They lie on the boundary of many discs, but never in the interior.

Proofs for the two well-known observations below are sketched. See Section 7.2 of \cite{elstrodt} for details.

\begin{lemma}\label{lem:discrete}The set of singular points in $K$ is discrete.\end{lemma}

\begin{proof}For a singular point $\zeta$, let $\beta\in\cO$ be such that $\beta(\zeta,1)$ is integral with minimal norm among integral ideals in the same class. Then $\beta$ must have minimal magnitude among nonzero elements of $\beta(\zeta,1)$ to avoid contradicting Definition \ref{def:sing}. Now, there are finitely many ideal classes, each with finitely many integral ideals of minimal norm, each with finitely many nonzero elements of minimal magnitude like $\beta$. So since $\zeta=\beta\zeta/\beta\in(\beta)^{-1}$, we see that all singular points are contained in a finite union of fractional ideals, each of which forms a lattice in $\bC$.\end{proof}

\begin{lemma}\label{lem:sing}$\zeta\in K$ is not a singular point if and only if $|\mu\zeta-\lambda|<1$ for coprime $\lambda,\mu\in\cO$ with $|\mu|\leq\cS(\Gamma)$.\end{lemma}

\begin{proof}Given $\beta\in K$, the existence of $\alpha\in K$ for which $(\zeta,1)=(\alpha,\beta)$ is equivalent to $\beta=\mu\zeta-\lambda$ for coprime $\lambda,\mu\in\cO$. So by Definition \ref{def:sing}, $\zeta$ is not singular if and only if $|\mu\zeta-\lambda|<1$ for coprime $\lambda,\mu\in\cO$. This means $\zeta$ lies in the projection of $S_{\lambda,\mu}$. But then projections of the faces of $\cB$ also cover $\zeta$. So we may take $|\mu|\leq\cS(\Gamma)$.\end{proof}

\begin{figure}
    \centering
    \includegraphics[trim=3.0cm 2cm 3.1cm 1.6cm,clip,height=4.89cm]{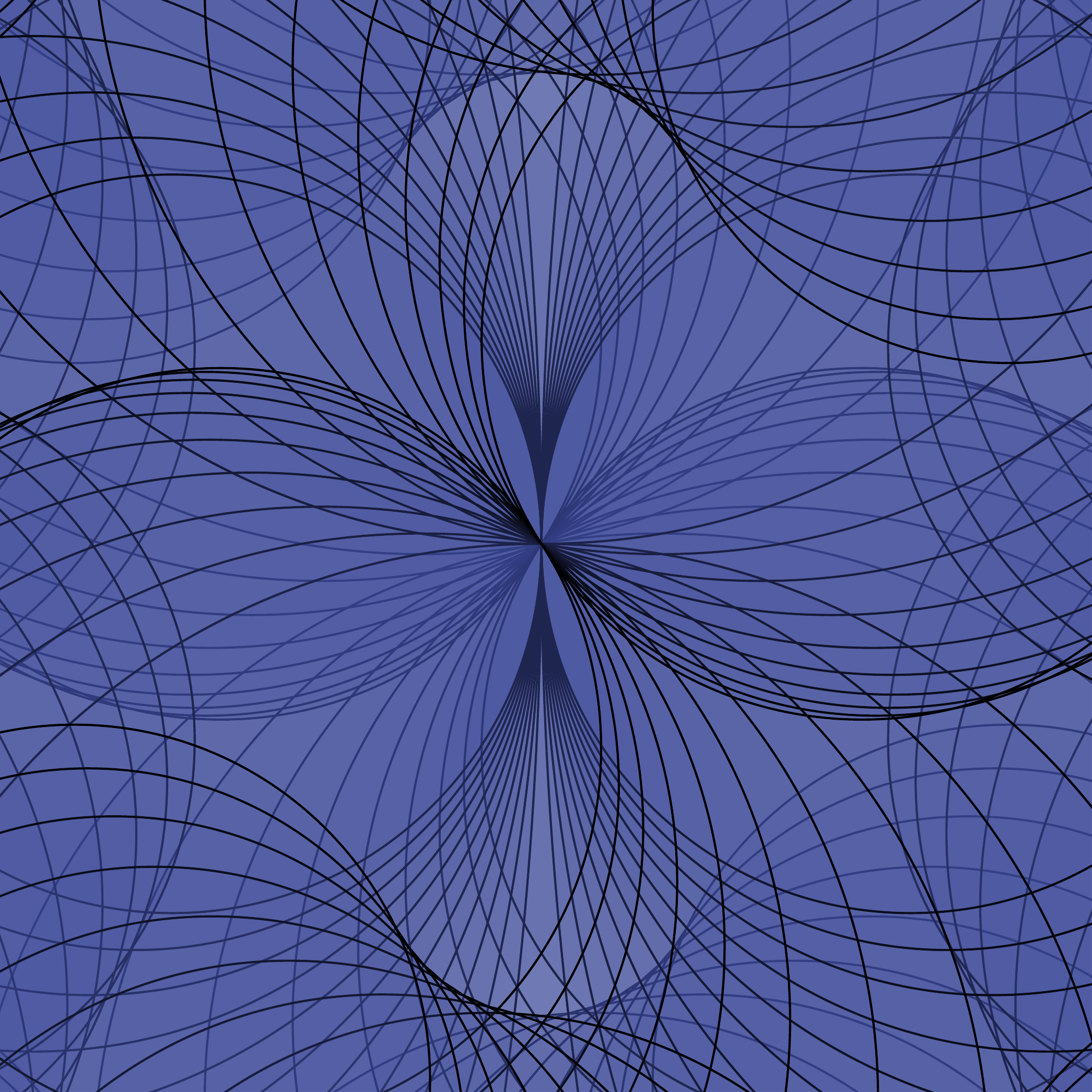}\hspace{0.16cm}
    \includegraphics[trim=2.2cm 0.8cm 2.2cm 0.8cm,clip,height=4.89cm]{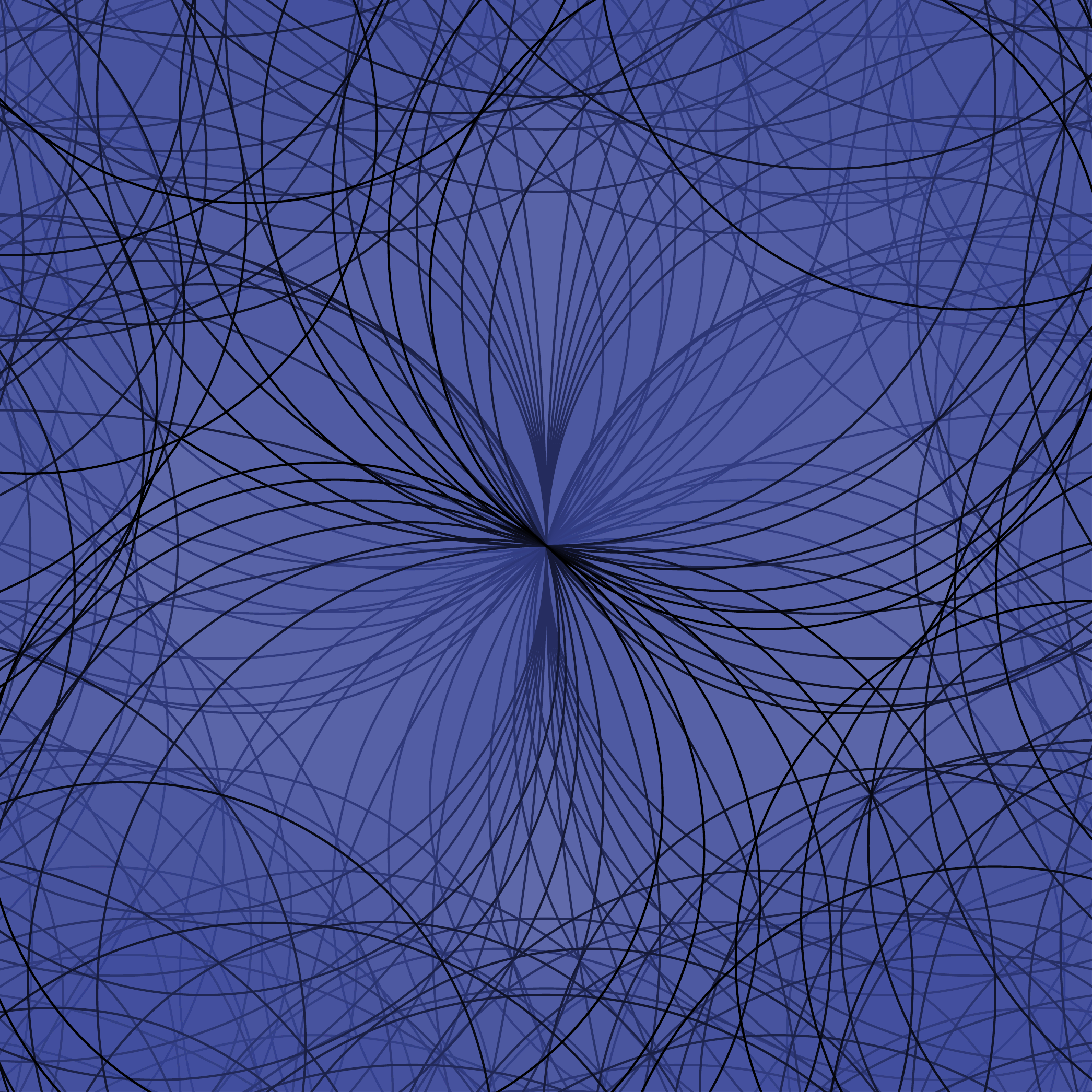}\hspace{0.16cm}
    \includegraphics[trim=3.6cm 0cm 3.6cm 0cm,clip,height=4.89cm]{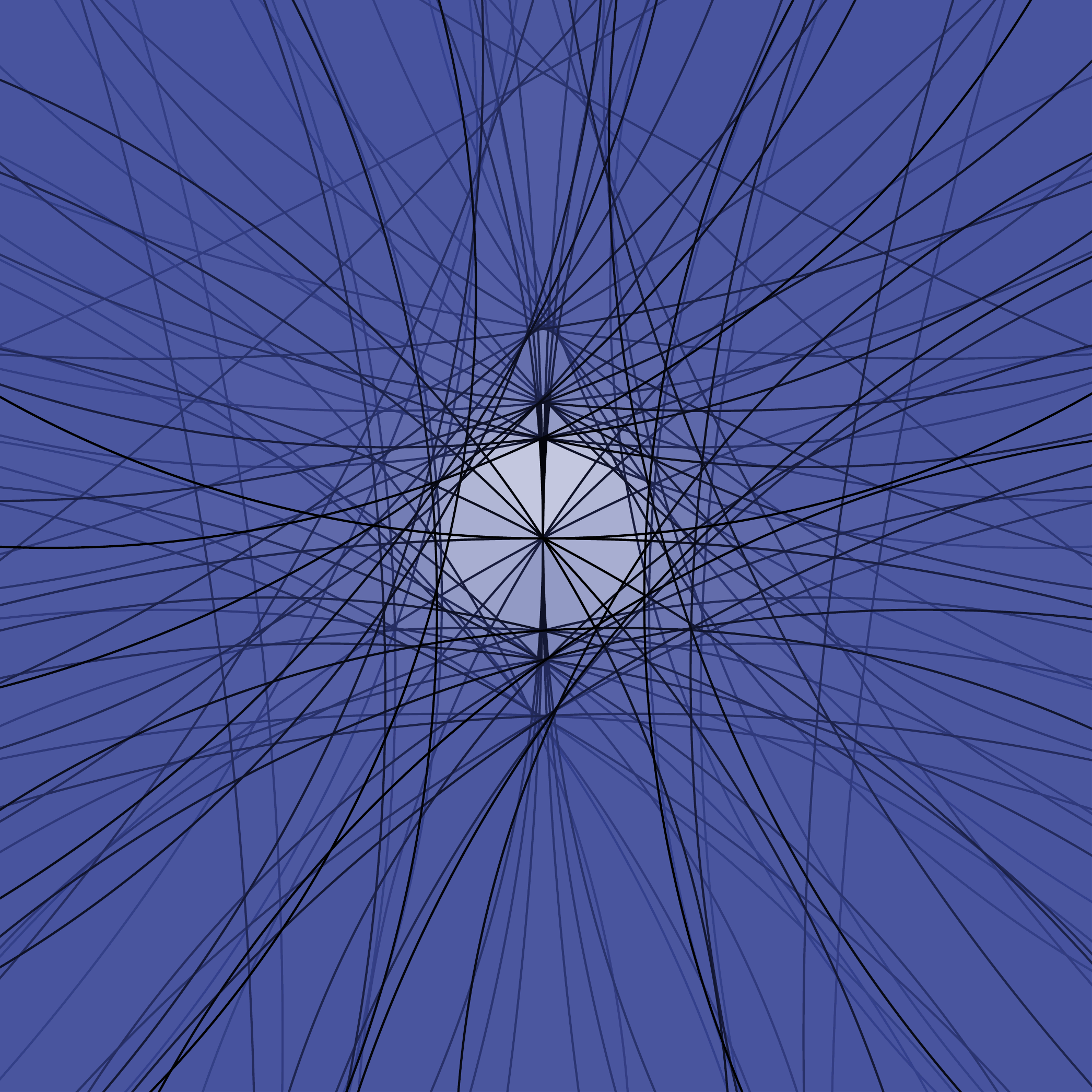}
    \captionsetup{width=0.967\textwidth}
    \caption{Projections of all $S_{\lambda/\mu}$ with $|\mu|\leq\cS(\Gamma)=\sqrt{528}$ around the singular points $(1+\sqrt{-33})/2$ (left), $\sqrt{-33}/3$ (middle), and $(3+\sqrt{-33})/6$ (right).}
    \label{fig:2}
\end{figure}

The next result is our best lower bound for $\cS(\Gamma)$ when $|\Delta|$ is prime or the product of two primes $p<p'$ with $p$ at least $3p'/64$. Otherwise Theorem \ref{thm:low} is stronger. 

\begin{proposition}\label{prop:low} $\cS(\Gamma)>(\sqrt{|\Delta|}-2)/\sqrt{3}$.\end{proposition}

\begin{proof}First observe that $((\sqrt{|\Delta|}-2)/\sqrt{3})^2\not\in\bZ$ unless $|\Delta|$ is a perfect square. This is only true when $\Delta=-4$, in which case the proposition holds trivially. For other discriminants, $|\mu|= (\sqrt{|\Delta|}-2)/\sqrt{3}$ is impossible, allowing us to ignore the strictness of this proposition's inequality. 

Now, let $\mu\in\cO$ with $0<|\mu|< (\sqrt{|\Delta|}-2)/\sqrt{3}$, and let $\zeta_6=e^{i\pi/3}$. We claim that $|\mu \zeta_6-\lambda|\leq 1$ for $\lambda\in\cO$ only if $\lambda/\mu$ is $0$ or $1$. Once proved, $|\zeta_6-\lambda/\mu| > 1/|\mu|$ gives an open set around $\zeta_6$ that is disjoint from any $S_{\lambda/\mu}$ with $|\mu|<(\sqrt{|\Delta|}-2)/\sqrt{3}$ and $\lambda/\mu\neq 0,1$. The two unit discs centered on $0$ and $1$ have $\zeta_6$ on their boundary, so these discs only cut off the bottom portion of our open set. By Lemma \ref{lem:discrete}, what remains contains some $\zeta\in K$ which is not singular. The claim follows by applying Lemma \ref{lem:sing} to such a $\zeta$.

If $\mu\in\bZ$ then $|\Im(\mu \zeta_6)| = \sqrt{3}|\mu|/2<\sqrt{|\Delta|}/2 - 1$. Thus the only multiple of $\sqrt{|\Delta|}/2$ that $\Im(\lambda)$ might equal to achieve $|\mu \zeta_6-\lambda|\leq 1$ is $0$. So $\lambda\in\bZ$, which forces $|\mu|=1$ to achieve $|\Im(\mu \zeta_6-\lambda)|=\sqrt{3}|\mu|/2\leq 1$, in turn forcing $\lambda/\mu=0$ or $1$.

If $\mu\not\in\bZ$ then $|\Im(\mu)|=\sqrt{|\Delta|}/2$. Combined with $|\mu|<(\sqrt{|\Delta|}-2)/\sqrt{3}$, this gives \begin{equation}|\Re(\mu)|< \frac{\sqrt{|\Delta|}-4}{2\sqrt{3}}.\label{eq:2}\end{equation} Also, $|\mu \zeta_6|=|\mu|<\sqrt{|\Delta|}-1$, so we can only hope to achieve $|\mu \zeta_6-\lambda|\leq 1$ if $\Im(\lambda)$ is $0$ or $\pm\sqrt{|\Delta|}/2$. In particular, $|\Im(\mu)\Re(\zeta_6)-\Im(\lambda)|=|\pm\!\sqrt{|\Delta|}/4-\Im(\lambda)|\geq\sqrt{|\Delta|}/4$. But then $$|\Im(\mu \zeta_6-\lambda)|=|\Re(\mu)\Im(\zeta_6)+\Im(\mu)\Re(\zeta_6)-\Im(\lambda)|\geq -\frac{\sqrt{3}|\Re(\mu)|}{2}+\frac{\sqrt{|\Delta|}}{4}> 1,$$ where the last inequality uses (\ref{eq:2}).\end{proof}

\begin{theorem}\label{thm:low}If $\delta\in\cO$ has maximal magnitude among proper divisors of $\Delta$, then $\cS(\Gamma)>|\delta|/8$.\end{theorem}

\begin{proof}If $|\delta|/8\leq (\sqrt{|\Delta|}-2)/3$, the claim follows directly from Proposition \ref{prop:low}. This inequality holds when $|\delta|\leq 4\sqrt{|\Delta|}$ provided $|\Delta|\geq 223$. Theorem \ref{thm:low} can be verified manually for $|\Delta|<223$ (data for $|\Delta|<400$ is displayed in Figure \ref{fig:6}), so we assume for the proof that $|\delta| > 4\sqrt{|\Delta|}$.

Outside of rational integers, the only divisors of $\Delta$ are $\sqrt{\Delta}$ and, when $\Delta$ is even, $\sqrt{\Delta}/2$. Therefore $|\delta| > 4\sqrt{|\Delta|}$ forces $\delta=\pm\Delta/p$, where $p$ is the smallest rational prime dividing $\Delta$. Let $\fp$ be the prime over $p$, and let $\pi$ be $-\sqrt{\Delta}/2$ or $(p-\sqrt{\Delta})/2$ so that $\pi\in\fp$. Then $\bZ[\pi]=\cO$ and $p^2\nmid |\pi|^2$. Take $a\in\bZ$ with $a\equiv p^{-1}\, \text{mod}\,|\pi|^2/p$ and set $\zeta=ap/\pi$. 

Our first assertion is that $\zeta$ is not a singular point. To see this, note that the smallest nonzero elements in the ideal $(ap,\pi)=(p,\pi)=\fp$ are $\pm p$. So $\zeta$ is a singular point if and only if $ap/\pi=\alpha/p$ for some $\alpha\in\fp$. But comparing powers of $\fp$ on either side of $ap^2=\alpha\pi$ shows that $\fp^2=(p)$ divides $\alpha$, contradicting nontriviality of the ideal class $[(\alpha,p)]=[(\zeta,1)]$. Thus $\zeta$ is not a singular point. Therefore, by Lemma \ref{lem:sing} there are coprime $\lambda,\mu\in\cO$ with $|\mu|\leq\cS(\Gamma)$ and $|\mu\zeta-\lambda|<1$.

Now suppose $\lambda,\mu\in\cO$ are such that $0<|\mu|< |\delta|/8$ and $|\mu\zeta-\lambda|< 1$. Our second assertion is that $(\lambda,\mu)\subset\fp$. This will complete the proof as $8\nmid \delta$ makes $|\mu|=|\delta|/8$ impossible, allowing us to ignore the strictness of the theorem's inequality.

Write $\mu=b+c\pi$ for $b,c\in\bZ$. Since $$\Im(\mu\zeta)=\frac{ab\sqrt{|\Delta|}}{2|\pi|^2/p}$$ has to be within $1$ of $\Im(\lambda)$---a multiple of $\sqrt{|\Delta|}/2$---there must be a rational integer in the congruence class $ab\,\text{mod}\,|\pi|^2/p$ that has magnitude less than $2|\pi|^2/p\sqrt{|\Delta|}$. But then $a\equiv p^{-1}\, \text{mod}\,|\pi|^2/p$ means \begin{equation}\pm\left\{0,p,2p,...,\left\lfloor\frac{2|\pi|^2}{p\sqrt{|\Delta|}}\right\rfloor\!p\right\}\label{eq:3}\end{equation} is a complete set of possible representatives for $b\,\text{mod}\,|\pi|^2/p$. The initial assumption that $4\sqrt{|\Delta|}<|\delta|=|\Delta|/p$ bounds the elements in (\ref{eq:3}) by $|\pi|^2/2p$, so they are reduced congruence class representatives in magnitude. It is straightforward to check that $|b+c\pi|=|\mu|< |\delta|/8$ implies $|b|< |\delta|/8$, in turn giving $|b|<|\Delta|/8p\leq |\pi|^2/2p$. Therefore $b$ belongs to (\ref{eq:3}), and $\mu\in\fp$.

Next, note that $$\mu\zeta = \frac{ab\overline{\pi}}{|\pi|^2/p} + acp\equiv\frac{ab\overline{\pi}}{|\pi|^2/p}\equiv \frac{(b/p)\overline{\pi}}{|\pi|^2/p}= \frac{b}{\pi}\,\text{mod}\,\fp,$$ where the second congruence uses $\overline{\pi}\in\overline{\fp}=\fp$. So to show that only elements of $\fp$ are strictly within $1$ of $\mu\zeta$, it is enough to check that $0$ is the only integer strictly within $1$ of $b/\pi$. Using $|b|<2|\pi|^2/\sqrt{|\Delta|}$ (from (\ref{eq:3})) gives $|\Re(b/\pi)|,|\Im(b/\pi)|<1$, so it remains only to eliminate $\pm 1$ as nearby integers. But $|b/\pi\pm 1|<1$ if and only if $0<|b|< 2\Re(\pi)$, which no multiple of $p$ satisfies.\end{proof}

\begin{figure}
    \centering
    \begin{overpic}[trim=4cm 0cm 4cm 0cm,clip,height=6cm,unit=1mm]{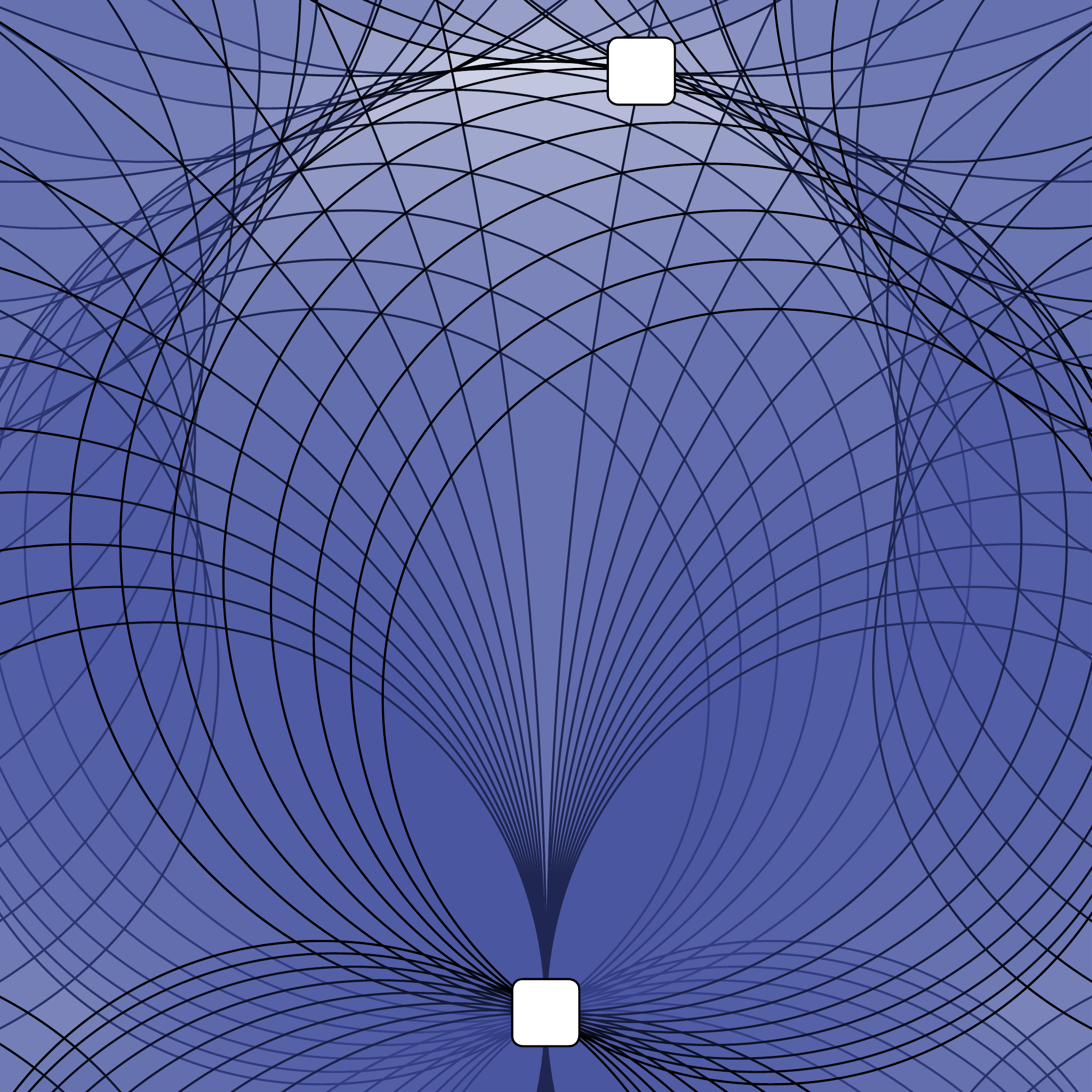}
    \put(17.88,3.2){$1$}
    \put(23.07,54.9){$2$}
    \end{overpic}\hspace{0.2cm}
    \begin{overpic}[trim=0.5cm 0.5cm 0.5cm 0.5cm,clip,height=6cm,unit=1mm]{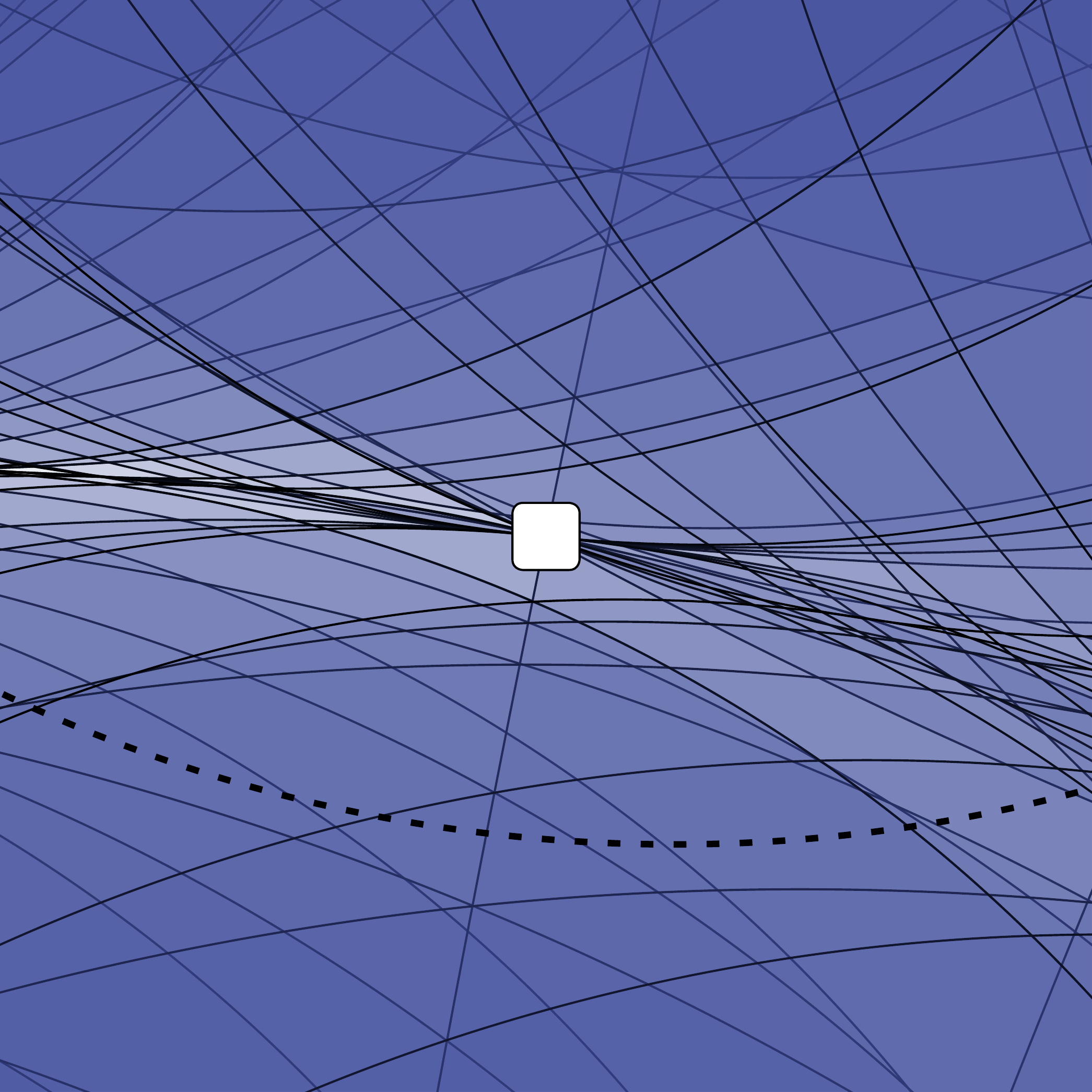}
    \put(29.17,29.4){$2$}
    \end{overpic}
    \captionsetup{width=0.79\textwidth}
    \caption{The singular point $(1+\sqrt{-97})/2$ (``1") and $\zeta$ from the proof, $25(1+\sqrt{97})/49$ (``2"), shown to be uncovered for $|\mu|<41$.}
    \label{fig:3}
\end{figure}

An example is shown in Figure \ref{fig:3} with $\Delta=-388$. Using notation from the proof, the smallest prime dividing $\Delta$ is $p=2$, so $\pi = 1-\sqrt{-97}$. Then $a=25$ solves $a\equiv 2^{-1}\,\text{mod}\,49$, giving $\zeta = 25(1+\sqrt{-97})/49$. This is the point labeled ``$2$."

The point labeled ``$1$" is a singular point for the ideal class of the prime over $2$: $\overline{\pi}/p=(1+\sqrt{-97})/2$. The first image shows how taking projections of $S_{\lambda/\mu}$ for $|\mu| \leq |\delta|/8 = 97/4$ have already surrounded the singular point. In fact, the largest curvature of a hemisphere that contributes a face of $\cB$ and has $(1+\sqrt{-97})/2$ on its boundary is $2\sqrt{97}$. On the other hand, the second image shows that $\zeta$ is not in the interior of a disc even up to $|\mu| < 41$, past the bound $97/4$ guaranteed by Theorem \ref{thm:low}. It is not until $\mu = 41$ that the projection of some $S_{\lambda/\mu}$ contains $\zeta$. This is shown as a dashed circle in the second image. It has center $(21+21\sqrt{-97})/41$.

We mention that $\cS(\Gamma)=45$ for $\bQ(\sqrt{-97})$.

\section{A Jacobsthal-type function}\label{sec:3}

In effort to put an upper bound on $\cS(\Gamma)$, suppose there is $S_{\zeta}$ for $\zeta\in K$ which we aim to show has too small a radius to contribute a face of $\cB$. The strategy is to find a pair of approximations to $\zeta$, say $\lambda/\mu$ and $\lambda'/\mu'$, for which $(\lambda,\mu,\lambda',\mu')$ has only principal prime divisors, if any. Then we look for $j\in\bZ$ such that $(j\lambda+\lambda',j\mu+\mu')$ is principal (as good as coprime because we can divide by a principal generator) and $(j\lambda+\lambda')/(j\mu+\mu')$ still approximates $\zeta$ well. Ideally the hemisphere of $(j\lambda+\lambda')/(j\mu+\mu')$ completely covers $S_\zeta$, eliminating it as a potential face of $\cB$.

The value of $j$ that makes $(j\lambda+\lambda',j\mu+\mu')$ principal is unpredictable. So $\lambda/\mu$ needs to be an approximation of such quality that it affords freedom to $j$. We want $|\mu|$ small enough so that $|j\mu+\mu'|$ does not make the resulting hemisphere too small to cover $S_{\zeta}$. But we also want $|\mu\zeta-\lambda|$ small enough so that $|j(\mu\zeta-\lambda)+(\mu'\zeta-\lambda')|$ does not make the resulting hemisphere too far from $S_{\zeta}$. It is not often that two approximations of such quality can be found, which is why it is not supposed that $\lambda'/\mu'$ can also handle a coefficient, as in $(j\lambda+j'\lambda')/(j\mu+j'\mu')$.

The amount of ``slack" $j$ requires of $\lambda/\mu$ is determined by the following function.

\begin{definition}\label{def:jacob}For $\fa\subseteq\cO$, define $\cj(\fa)$ as the minimum such that, given any $\alpha\in\cO$ for which $(2,\alpha)$ is not a split prime if $2\,|\,\fa$, every closed interval of length $\cj(\fa)$ contains $j\in\bZ$ with $(\fa,\alpha+j)=\cO$.\end{definition}

If $2$ splits and $\Im(\alpha)$ is an odd multiple of $\sqrt{|\Delta|}/2$, then $(2,\alpha+j)$ is prime for all $j\in\bZ$. Definition \ref{def:jacob} avoids this with the constraint on $\alpha$ when $2\,|\,\fa$.

An asymptotic bound on $\cj$ follows readily from a lower bound sieve of dimension two. The argument is included, though it is perhaps well-known to sieve theory.

\begin{proposition}\label{prop:asymptotic}For $\fa\subseteq\cO$, let $\omega(\fa)$ be the number of distinct rational primes dividing the norm $\|\fa\|$. Then $\cj(\fa)\ll \omega(\fa)^{4.27}$, where the implied constant does not depend on $\cO$.
\end{proposition}

\begin{proof}In order to make $\alpha+j$ and $\fa$ from Definition \ref{def:jacob} coprime, at most two congruence classes (for split primes) must be avoided by $j$ for each of the $\omega(\fa)$ rational primes dividing $\|\fa\|$. In dimension two, the Diamond-Halberstam-Richert lower bound sifting limit is $\beta_2=4.2664...$ \cite{diamond}. That is, for any $\varepsilon > 0$ there are constants $c,c'>0$ such that any interval of length $x> cp_n^{\beta_2+\varepsilon}$ contains at least $c'x/\log^2x$ integers avoiding any single congruence class $\text{mod}\,2$ and any pairs of congruence classes $\text{mod}\,3,5,...,p_n$ for any $n$. 

Let $0<\varepsilon'<4.27/\beta_2-1$ and fix any $\varepsilon > 0$ such that $(1+\varepsilon')(\beta_2+\varepsilon)<4.27$. Let $p_n$ be the largest prime less than $\omega(\fa)^{1+\varepsilon'}$, so $p_n^{\beta_2+\varepsilon}<\omega(\fa)^{4.27}$. In particular, if the primes dividing $\|\fa\|$ are all among the first $n$, then $\cj(\fa)< c\omega(\fa)^{4.27}$ as per the previous paragraph. We now consider the possibility that $p>p_n$ divides $\|\fa\|$.

The number of integers in an interval of length $x$ which belong to one of two congruence classes $\text{mod}\,p$ is at most $2+2x/p$. Thus the subset of such an interval that does not avoid two congruence classes per prime $p>p_n$ dividing $\|\fa\|$ has cardinality at most $$\sum_{\substack{p\,|\,\|\fa\| \\ p>p_n}}\!\left(2+\frac{2x}{p}\right)< 2\omega(\fa)+\frac{2x\omega(\fa)}{p_{n+1}}<2\omega(\fa)+\frac{2x\omega(\fa)}{\omega(\fa)^{1+\varepsilon'}}<\frac{4x}{\omega(\fa)^{\varepsilon'}}.$$ The last inequality holds with $x>\omega(\fa)^{1+\varepsilon'}$. 

Altogether, we have shown that if $4x/\omega(\fa)^{\varepsilon'}<c'x/\log^2 x$, there is at least one integer avoiding all required congruence classes modulo the prime divisors of $\|\fa\|$ in each interval of size $x> cp_n^{\beta_2+\varepsilon}$. So set $x=c\hspace{0.05cm}\omega(\fa)^{4.27}$. Then $4/\omega(\fa)^{\varepsilon'}<c'/\log^2(c\hspace{0.05cm}\omega(\fa)^{4.27})$ provided $\omega(\fa)$ is sufficiently large.\end{proof}

\begin{lemma}\label{lem:jacob}Suppose $\fa\subseteq\cO$ is coprime to $(\lambda,\mu,\lambda',\mu')\subseteq\cO$ and contains $\lambda\mu'-\mu\lambda'$, and that $(2,\lambda',\mu')=\cO$ or $(2,\lambda+\lambda',\mu+\mu')=\cO$ if $2$ splits and divides $\fa$. Then every closed interval of length $\cj(\fa)$ contains $j\in\bZ$ such that $(\fa,j\lambda+\lambda',j\mu+\mu')=\cO$.\end{lemma}

\begin{proof}Since $(\fa,\lambda,\mu,\lambda',\mu')=\cO$, only a prime ideal that divides $\fa$ but not $(\lambda,\mu)$ could potentially divide $(\fa,j\lambda+\lambda',j\mu+\mu')$ for some $j$. Given such a prime $\fp$, choose $\alpha\in\cO$ with $\alpha\equiv\lambda^{-1}\lambda'\,\text{mod}\,\fp$ if $\lambda\not\in\fp$ and $\alpha\equiv\mu^{-1}\mu'\,\text{mod}\,\fp$ if $\lambda\in\fp$. With this choice of $\alpha$, $(\fp,\alpha+j)=\cO$ implies $(\fp,j\lambda+\lambda',j\mu+\mu')=\cO$. The converse is also true: If $(\fp,\alpha+j)=\fp$ and $\lambda\not\in\fp$, for example, the choice of $\alpha$ gives $j\lambda+\lambda'\in\fp$. But then $\lambda(j\mu+\mu')-\mu(j\lambda+\lambda')=\lambda\mu'-\mu\lambda'\in\fa\subseteq\fp$ shows that $j\mu+\mu'\in\fp$, and thus $(\fp,j\lambda+\lambda',j\mu+\mu')=\fp$. The same argument works if $(\fp,\alpha+j)=\fp$ and $\mu\not\in\fp$.

If $2$ splits and $\alpha\,\text{mod}\,2$ has not already been fully determined in the previous paragraph, we ask that $(2,\alpha)$ is not prime in accordance with Definition \ref{def:jacob}. If $\alpha\,\text{mod}\,2$ is determined, this constraint is already met. Indeed, if $(2,\alpha)$ is a split prime then so is $(2,\alpha+j)$ for any $j\in\bZ$. But the previous paragraph showed that the same is then true for $(2,j\lambda+\lambda',j\mu+\mu')$, violating our hypothesis. Thus Definition \ref{def:jacob} gives the desired $j$.\end{proof} 

The hypothesis $\lambda\mu'-\mu\lambda'\in\fa$ loses no generality since $(j\lambda+\lambda',j\mu+\mu')$ always contains $\lambda\mu'-\mu\lambda'$. That is, prime ideals that do not contain $\lambda\mu'-\mu\lambda'$ are avoided by $(j\lambda+\lambda',j\mu+\mu')$ automatically.

From the start of this section, recall that it suffices for our purpose if $j\lambda+\lambda'$ and $j\mu+\mu'$ generate a principal ideal---they need not be coprime. This is why the function introduced below only sieves non-principal primes.

\begin{notation}\label{not:J}For $x>0$ let $\cJ(x)$ be the maximum obtained by $\cj(\fa)$ over $\fa\subset\cO$ that have no principal prime divisors and satisfy $\sqrt{\|\fa\|}<x$.\end{notation}

The asymptotic bound on $\cJ$ resulting from Proposition \ref{prop:asymptotic} can be found in Corollary \ref{cor:asymptotic}. An explicit bound on $\cj$ or $\cJ$ can also be obtained with slight modification to the argument made by Stevens \cite{stevens} for Jacobsthal's function, but it is too large to be computationally useful.

For the purpose of bounding Swan's number and computing Bianchi polyhedra, values of $\cJ$ can just be calculated to make the upper bound in Theorem \ref{thm:up} explicit. This requires polynomial time, whereas an algorithm to compute $\cB$ like Bianchi.gp \cite{bianchigp} has, according to Rahm, complexity at least $O(\text{exp}(|\Delta|^{3/2}))$. Even without the use of sieves to evaluate $\cJ$ when making Figure \ref{fig:6}, the time required was negligible.

\begin{lemma}\label{lem:low}Let $\delta\in\cO$ have maximal magnitude among proper divisors of $\Delta$. If $K$ has nontrivial class group, then $\cJ(2|\delta|)\geq 3$. If, additionally, $2$ splits and there is a non-principal prime of norm strictly between $2$ and $|\delta|$, then $\cJ(2|\delta|)\geq 6$.\end{lemma}

\begin{proof}For the first claim, we will find distinct non-principal primes $\fp$ and $\fq$ with $\sqrt{\|\fp\fq\|}<2|\delta|$. In Definition \ref{def:jacob}'s notation, we then choose $\alpha\in\fp$ with $\alpha\equiv 1\,\text{mod}\,\fq$, and center an interval of length less than $3$ on $-1/2$ so that the only intergers $j$ it might contain are $0$ and $-1$. Since $(0+\alpha,\fp\fq)=\fp$ and $(-1+\alpha,\fp\fq)=\fq$, this would prove $\cj(\fp\fq)\geq 3$ as long as $\fp$ and $\fq$ are not both of norm $2$ as per Definition \ref{def:jacob}. 

The second claim is similar, but now neither $\fp$ nor $\fq$ can have norm $2$, and we will need $\sqrt{\|\fp_2\fp\fq\|}<2|\delta|$, where $\fp_2$ is a split prime over $2$. Then take $\alpha\in\fp_2$ with $\alpha\equiv -1\,\text{mod}\,\fp$ and $\alpha\equiv 1\,\text{mod}\,\fq$. Centering an interval of length less than $6$ on $0$ forces $j\in\{0,\pm 1,\pm 2\}$. We have $(j+\alpha,\fp_2\fp\fq)\subseteq\fp_2$ for $j\in\{0,\pm2\}$, and $(-1+\alpha,\fp_2\fp\fq)=\fq$ and $(1+\alpha,\fp_2\fp\fq)=\fp$. Thus $\cj(\fp_2\fp\fq)\geq 6$.

First suppose $|\Delta|$ has at least two prime divisors, say $p<q$, and let $\fp$ and $\fq$ be primes above them. Suppose further that $q\neq |\Delta|/4$, so $\fq$ is not principal. The first claim follows. For the second, if $2$ splits then $\|\fp\|,\|\fq\|\neq 2$ since $p$ and $q$ ramify.

If $|\Delta|$ is prime, there are no non-principal ramified primes. So a nontrivial class group along with Minkowski's bound gives us a non-principal split prime $\fp$ with norm at most $\sqrt{|\Delta|/3}$. Then $\fp$ and $\overline{\fp}$ serve as our two primes to prove the first claim. For the second claim, the lemma assumes we may take $\|\fp\|\neq 2$.

If $|\Delta|/4$ is prime, there is one non-principal ramified prime: $(2,1+\sqrt{\Delta}/2)$. Its ideal class has order two. In particular, if the class number of $K$ exceeds two, there is a nontrivial ideal class unaccounted for by powers of $(2,1+\sqrt{\Delta}/2)$. Combined with Minkowski's bound, this gives $\fp$ as in the previous paragraph. There are three outstanding cases with class number two and $|\Delta|/4$ prime: $\Delta = -20$, $-52$, and $148$. All have non-principal split primes with norm (much) smaller than $2|\delta|$.\end{proof}

\section{An upper bound for Swan's number}\label{sec:4}

The next two lemmas are used in the proof of Theorem \ref{thm:up} to produce $\lambda/\mu$ and $\lambda'/\mu'$ from Lemma \ref{lem:jacob} and the discussion at the start of Section \ref{sec:3}.

The first lemma refers to continued fraction convergents (or approximations). The floor function continued fraction algorithm begins with a point to be approximated, say $z=z_0\in\bR$. Then it proceeds by setting $z_{n+1}=1/(z_n-\lfloor z_n\rfloor)$ for $n\geq 0$. The convergents $p_n/q_n$ start with $p_{-1}=q_0=0$ and $q_{-1}=p_0=1$, and continue with $p_{n+1} = \lfloor z_n\rfloor p_n+p_{n-1}$ and $q_{n+1} = \lfloor z_n\rfloor q_n+q_{n-1}$. These variables satisfy the following.

\begin{lemma}\label{lem:cfrac}For $n\geq 1$, $\text{gcd}(q_{n-1},q_n)=1$, $|q_nz-p_n|=(-1)^n(p_n-q_nz)$, and $$|q_nz-p_n| = \frac{|q_{n-1}z-p_{n-1}|}{z_n}= \frac{1}{z_nq_n+q_{n-1}}\leq\frac{1}{q_{n+1}}.$$\end{lemma}

\begin{proof}See Section 1.3 of \cite{hensley}, for example.\end{proof}

Regarding the next lemma, there may already be results related to variants of Gauss' circle problem that imply even stronger inequalities. For this reason, and because the arguments are far from new, the proof is only sketched. 

\begin{lemma}\label{lem:thue}If $r\geq 4\sqrt{|\Delta|}$, then for any $\zeta\in\bC$ there exist $\lambda,\mu\in\cO$ with $0<|\mu|<r$ and $|\mu\zeta-\lambda|<\sqrt{|\Delta|}/\sqrt{2}r$.\end{lemma}

\begin{proof}Consider the circle $C$ of radius $r/2$ centered at $0\in\bC$. Trace the center of a rectangle $R$ of width $1$ and height $\sqrt{|\Delta|}/2$ as it rotates around $0$ so that at least one corner always touches $C$ and no corner lies outside $C$. The area inside the traced shape is $$A =\frac{\pi r^2}{4}+\frac{\sqrt{|\Delta|}}{2}-\frac{\sqrt{r^2-1}}{2}-\frac{\sqrt{4|\Delta|r^2-\Delta^2}}{8}-\frac{r^2}{2}\arcsin\!\left(\frac{1}{r}\right)-\frac{r^2}{2}\arcsin\!\left(\frac{\sqrt{|\Delta|}}{2r}\right).$$ To verify this, observe that in the first quadrant of the plane, where it is the top-right corner of $R$ that touches $C$, we have traced over a shifted (by $-1/2-\sqrt{\Delta}/4$) copy of $C$. The portion of this shifted circle that lies in the first quadrant can be computed with standard geometric formulas to find $A/4$. 

When $R$ is centered on some $\mu\in\cO$, it intersects our traced shape if and only if $\mu\in C$. Since copies of $R$ centered on integers tile the plane and each has area $\sqrt{|\Delta|}/2$, we have $|\cO\cap C| > 2A/\sqrt{|\Delta|}$ (Gauss' argument \cite{clark}). Using Dirichlet's box principle and Thue's circle packing theorem \cite{thue}, when $|\cO\cap C|$ exceeds $2r^2/\sqrt{3|\Delta|}$ there are distinct $\mu_1,\mu_2\in \cO\cap C$ such that $|(\mu_1\zeta-\lambda_1)-(\mu_2\zeta-\lambda_2)|<\sqrt{|\Delta|}/\sqrt{2}r$ for the appropriate choices of $\lambda_1$ and $\lambda_2$. Then $\lambda=\lambda_1-\lambda_2$ and $\mu=\mu_1-\mu_2$ is the desired pair. It can be checked that $2A/\sqrt{|\Delta|}>2r^2/\sqrt{3|\Delta|}$ for $r \geq 4\sqrt{|\Delta|}$.\end{proof} 

\begin{figure}[H]
    \centering
    \includegraphics[trim=1.6cm 0.9cm 1.6cm 0.7cm,clip,height=6cm]{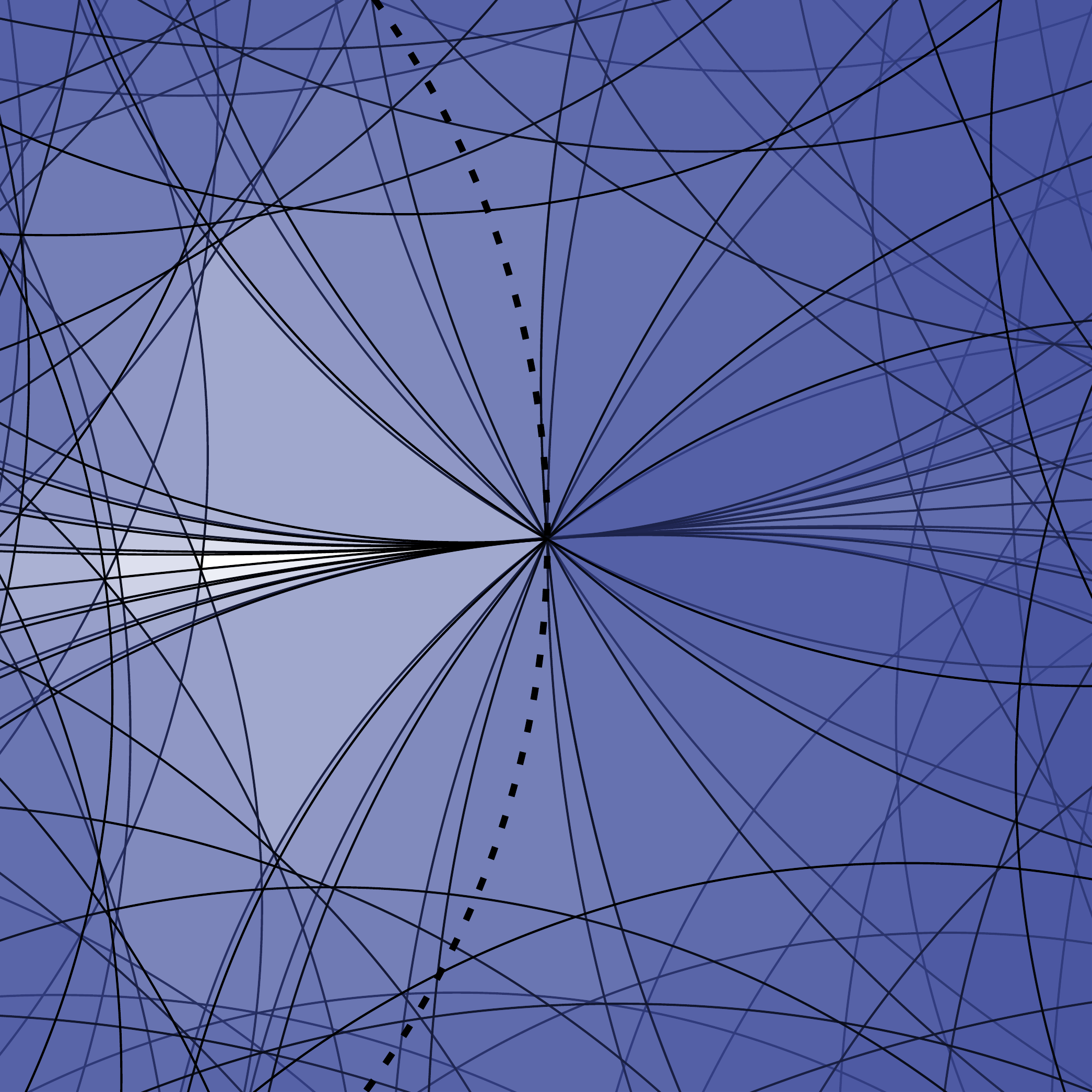}\hspace{0.28cm}\includegraphics[trim=3.5cm 0.7cm 3.5cm 0.7cm,clip,height=6cm]{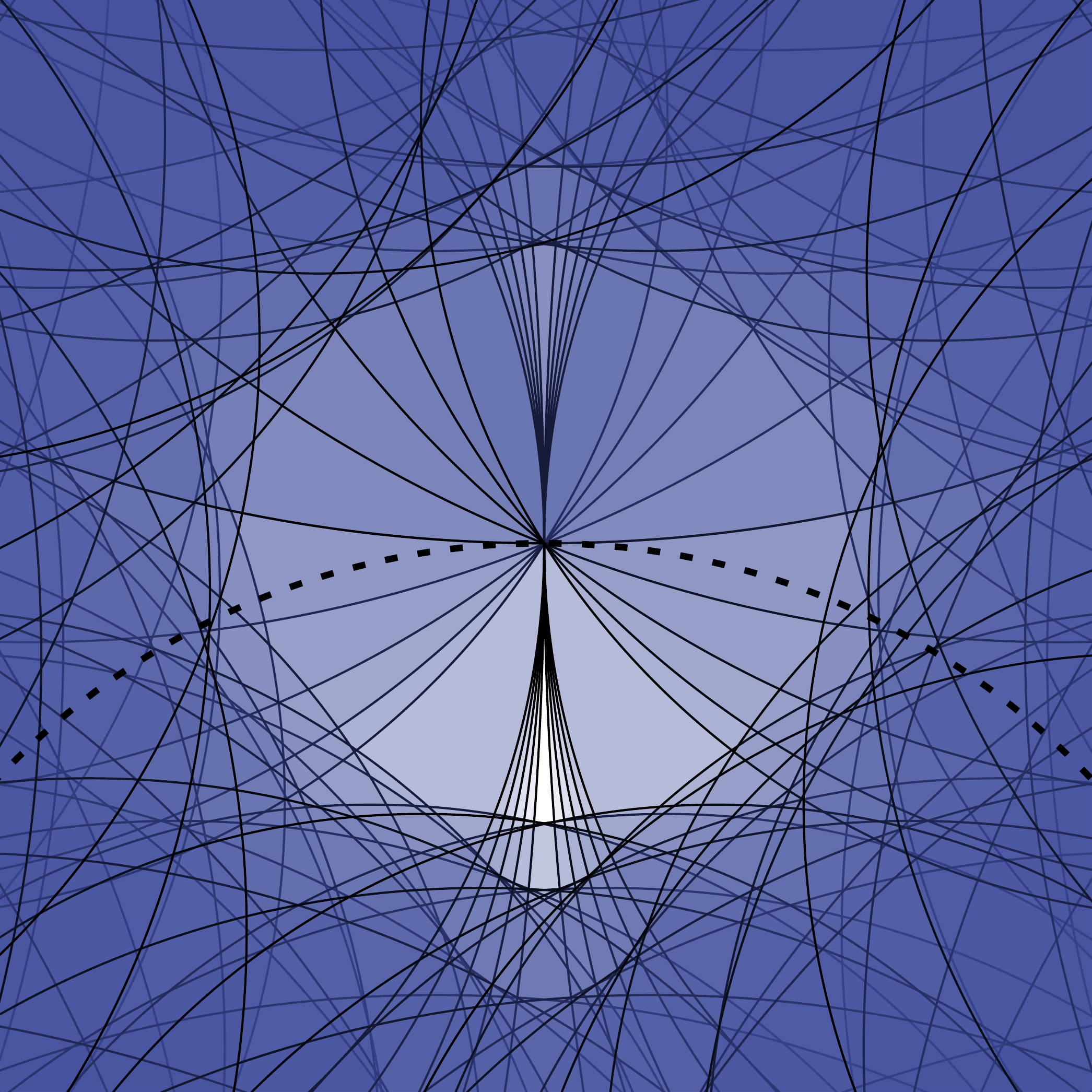}
    \captionsetup{width=0.8\textwidth}
    \caption{The relative angle of the thin, uncovered gap to its singular point determines how our proof finds the dashed circle.}
    \label{fig:4}
\end{figure}

The proof of Theorem \ref{thm:up} is divided into cases depending on the kinds of approximations that can be found near some nonsingular $\zeta\in K$. The end of each case is marked with the symbol $\diam$.

Let us outline the proof loosely, referencing notation and discussion from the start of Section \ref{sec:3}. Cases \hyperref[case:1]{1} and \hyperref[case:2]{2} find $\lambda/\mu$ and $\lambda'/\mu'$ when $\zeta$ does not have an approximation of especially high quality. But if, for example, $\lambda/\mu$ happens to be a singular point very near $\zeta$, then it is a challenge to find the second pair $\lambda',\mu'$ such that $|\mu'\zeta-\lambda'|$ and $|\mu'|$ are reasonably small---rational numbers (like $\lambda/\mu$ and the sought-after $\lambda'/\mu'$) repel one another, as usual in Diophantine approximation. The way to proceed depends on the argument of $\zeta-\lambda/\mu$.

Figure \ref{fig:4}'s images are centered on those $\lambda/\mu$ which are found in forthcoming proof when $\zeta$ lies in either of the thin, uncovered gaps. In the first image, $\lambda/\mu$ is the singular point $(-1+\sqrt{-143})/12$. The uncovered gap protrudes almost horizontally from $\lambda/\mu$, which makes this scenario amenable to Case \hyperref[case:3]{3}'s arguments. Here we use $\lambda'=1$ and $\mu'=0$, and it turns out that $(-4\lambda+\lambda',-4\mu+\mu') =\cO$. The dashed circle has center $(-4\lambda+\lambda')/(-4\mu+\mu')= (-3+2\sqrt{-143})/24$ and curvature $24$, which happens to be $\cS(\Gamma)$.

In the second image, $\lambda/\mu$ is the singular point $(3+\sqrt{-57})/6$. The uncovered gap protrudes vertically from $\lambda/\mu$, which means Case \hyperref[case:3]{3} is not applicable. The strategy given immediately following Case \hyperref[case:3]{3} in the proof finds $\lambda'=-8+\sqrt{-57}$ and $\mu' = 3+\sqrt{-57}$. Since $\lambda'/\mu'=\lambda/\mu$, we are in Case \hyperref[case:5]{5} rather than Case \hyperref[case:4]{4}. Under Case \hyperref[case:5]{5}'s direction we find that $(1-2\lambda+4\lambda',-2\mu+4\mu') =\cO$. The dashed circle has center $(1-2\lambda+4\lambda')/(-2\mu+4\mu')= (-37+2\sqrt{-57})/4\sqrt{-57}$ and curvature $\sqrt{912}$, which happens to be $\cS(\Gamma)$.

\begin{theorem}\label{thm:up}If $\delta\in\cO$ has maximal magnitude among proper divisors of $\Delta$ and $J\in\bZ$ is such that $\cJ(2\max(|\delta|,J\!\sqrt{|\Delta|}))\leq J$, then $\cS(\Gamma)<14J\max(|\delta|,J\!\sqrt{|\Delta|})$.\end{theorem}

\begin{proof}The claim is checked with data from \cite{aranes} or \cite{rahm} for the eight principal ideal domains. For other rings, Lemma \ref{lem:low} gives $J\geq 3$.

Let $d=\max(|\delta|,J\!\sqrt{|\Delta|})$. Fix $\zeta\in K$ with $(\zeta,1)$ principal such that $S_{\zeta}$ has curvature at least $14Jd$. We will show that $S_{\zeta}$ does not contribute a face of $\cB$. 

Using $r=2\sqrt{2}d$ in Lemma \ref{lem:thue} gives $\lambda,\mu\in\cO$ with \begin{equation}0<|\mu|<2\sqrt{2}d\hspace{0.5cm}\text{and}\hspace{0.5cm}|\mu\zeta-\lambda|<\frac{\sqrt{|\Delta|}}{4d}.\label{eq:4}\end{equation} Reduce $(\lambda,\mu)$ if possible so that it has minimal norm among nonzero integral ideals in its class. If $(\lambda,\mu)=\cO$, we claim $S_{\lambda/\mu}$ covers $S_{\zeta}$ and we are done. Indeed, the radius of $S_{\lambda/\mu}$ is $1/|\mu|$, so it suffices to check that $1/|\mu|-|\zeta-\lambda/\mu|$ is at least the radius of $S_{\zeta}$. The latter is bounded above by $1/14Jd$. We rearrange the desired inequality as \begin{equation}0<\frac{|\mu|}{1-|\mu\zeta-\lambda|}\leq 14Jd,\label{eq:5}\end{equation} which follows from (\ref{eq:4}). So suppose $(\lambda,\mu)\neq\cO$.

Consider the floor function continued fraction expansion of $2\Im(\zeta)/\sqrt{|\Delta|}$. Let $p_n/q_n$ be the first convergent satisfying \begin{equation}\left|\frac{2\Im(\zeta)q_n}{\sqrt{|\Delta|}}-p_n\right|<\frac{1}{d}.\label{eq:6}\end{equation} Then Lemma \ref{lem:cfrac} (with index decreased by one) gives the first inequality below and minimality of $n$ gives the second: \begin{equation}q_n < \left|\frac{2\Im(\zeta)q_{n-1}}{\sqrt{|\Delta|}}-p_{n-1}\right|^{-1}\leq d.\label{eq:7}\end{equation}

\phantomsection\label{case:1}\noindent\textbf{Case 1:} $|2\Im(\zeta)q_n/\sqrt{|\Delta|}-p_n|\geq 1/2Jd$.\hspace{\parindent}If $(q_n,\mu)$ has only principal prime divisors, let $\mu'=q_n$ and $\lambda'\in\cO$ be nearest to $\mu'\zeta$. If a non-principal prime $\fp$ divides $(q_n,\mu)$, we aim to define $\mu'$ as a linear combination of $q_n$ and $q_{n+1}$, where $p_{n+1}/q_{n+1}$ is the next continued fraction convergent of $2\Im(\zeta)/\!\sqrt{|\Delta|}$. Let $\fa$ have maximal norm among ideals containing $\mu$ that have no principal prime divisors and are coprime to $(q_n)$. In particular, $\fp\nmid\fa$. Choose any integer $\alpha\equiv q_n^{-1}q_{n+1}\,\text{mod}\,\fa$. If needed, we can take $\alpha\in\bZ$ to meet Definition \ref{def:jacob}'s constraint when $2\,|\,\fa$. By definition, every interval of length $\cj(\fa)$ contains $j'\in\bZ$ satisfying $(\fa,\alpha+j')=\cO$. For such a $j'$, $(\fa,j'q_n+q_{n+1})=\cO$ by choice of $\fa$ and $\alpha$. Since $\text{gcd}(q_n,q_{n+1})=1$, this in turn shows that $(\mu,j'q_n+q_{n+1})$ has only principal prime divisors. Finally, using (\ref{eq:4}) for the third inequality below, we have $$\cj(\fa)\leq\cJ(\sqrt{\|\fa\|})\leq\cJ\!\bigg(\frac{|\mu|}{\sqrt{\|\fp\|}}\bigg)\leq\cJ\!\bigg(\frac{2\sqrt{2}d}{\sqrt{2}}\bigg)\leq J.$$ 

The interval of length $J$ from which we take $j'$ is centered at $1/z_{n+1}$, where $z_{n+1}$ is from our continued fraction expansion of $2\Im(\zeta)/\!\sqrt{|\Delta|}$. That is, we choose $j'$ so that $|j'-1/z_{n+1}|\leq J/2$. Let $\mu'=j' q_n+q_{n+1}$ and let $\lambda'\in\cO$ be nearest to $\mu'\zeta$. Since one of $p_n/q_n$ and $p_{n+1}/q_{n+1}$ is an over-approximation of $2\Im(\zeta)/\!\sqrt{|\Delta|}$ and the other is an under-approximation by Lemma \ref{lem:cfrac}, we have $$|\Im(\mu'\zeta-\lambda')|=\left|j'\bigg|q_n\Im(\zeta)-\frac{p_n\sqrt{|\Delta|}}{2}\bigg|-\bigg|q_{n+1}\Im(\zeta)-\frac{p_{n+1}\sqrt{|\Delta|}}{2}\bigg|\right|=$$ \begin{equation}\left|j'-\frac{1}{z_{n+1}}\right|\cdot\bigg|q_n\Im(\zeta)-\frac{p_n\sqrt{|\Delta|}}{2}\bigg|<\frac{J}{2}\cdot\frac{\sqrt{|\Delta|}}{2d}\leq\frac{1}{4}.\label{eq:8}\end{equation} Note that an even stronger upper bound on $|\Im(\mu'\zeta-\lambda')|$ is given by (\ref{eq:6}) if we took $\mu'=q_n$ and $\Im(\lambda')=p_n\sqrt{|\Delta|}/2$ at the start of Case \hyperref[case:1]{1}.

Increasing the index by one in Lemma \ref{lem:cfrac} shows $|2q_n\Im(\zeta)/\!\sqrt{|\Delta|}-p_n|/z_{n+1}=1/(z_{n+1}q_{n+1}+q_n)$. In particular, $$q_{n+1}=\frac{1}{|2q_n\Im(\zeta)/\sqrt{|\Delta|}-p_n|}-\frac{q_n}{z_{n+1}}.$$ Thus from (\ref{eq:7}) and the defining inequality of Case \hyperref[case:1]{1}, we have \begin{equation}|\mu'|=\left|\bigg(j'-\frac{1}{z_{n+1}}\bigg)q_n+\frac{1}{|2q_n\Im(\zeta)/\!\sqrt{|\Delta|}-p_n|}\right|\leq \frac{Jd}{2}+2Jd=\frac{5Jd}{2}.\label{eq:9}\end{equation}

Since $(\lambda,\mu)$ is not principal, neither is $\mu/(\lambda,\mu)$. This means $\mu/(\lambda,\mu)$ has a non-principal prime divisor, which does not divide $\mu'$ by choice of $j'$. Thus $\lambda\mu'\neq\mu\lambda'$. Next, if $2$ splits then $(2,\mu,\mu')=\cO$. So $\mu'\in\bZ$ implies either $(2,\mu')=\cO$ or $(2,\mu+\mu')=\cO$, thereby meeting the last of Lemma \ref{lem:jacob}'s hypotheses. From (\ref{eq:4}), (\ref{eq:8}), and (\ref{eq:9}) along with $|\Re(\mu'\zeta-\lambda'|\leq1/2$, we have $$|\lambda\mu'-\mu\lambda'|\leq |\mu'(\mu\zeta-\lambda)|+|\mu(\mu'\zeta-\lambda')|<\frac{5J\!\sqrt{|\Delta|}}{8}+2\sqrt{2}d\sqrt{\frac{1}{4}+\frac{1}{16}}<2\sqrt{2}d.$$ So if $\fa$ has maximal norm among ideals dividing $\lambda\mu'-\mu\lambda'$ that have no principal prime divisors and are coprime to $(\lambda,\mu)$, then $\cj(\fa)\leq\cJ(|\lambda\mu'-\mu\lambda'|/\!\sqrt{\|(\lambda,\mu)\|})\leq J$. Therefore Lemma \ref{lem:jacob} gives $j\leq J/2$ such that $(j\lambda+\lambda',j\mu+\mu')$ has only principal prime divisors. We claim the hemisphere corresponding to $(j\lambda+\lambda')/(j\mu+\mu')$ covers $S_{\zeta}$. Just as with (\ref{eq:5}), it suffices to verify that $$0<\frac{|j\mu+\mu'|}{1-|(j\mu+\mu')\zeta-(j\lambda+\lambda')|}\leq 14Jd.$$ Using (\ref{eq:4}), (\ref{eq:8}), and (\ref{eq:9}) again, the expression above is at most $$\frac{J|\mu|/2+|\mu'|}{1-J|(\mu\zeta-\lambda)|/2-|\mu'\zeta-\lambda'|}< \frac{\sqrt{2}Jd+5Jd/2}{1-1/8-\sqrt{1/4+1/16}}<14Jd.$$ Thus $S_{\zeta}$ does not contribute a face of $\cB$.\hfill$\diam$

\hspace{\baselineskip}

Henceforth assume Case \hyperref[case:1]{1} fails, so \begin{equation}\left|q_n\Im(\zeta)-\frac{p_n\sqrt{|\Delta|}}{2}\right|<\frac{\sqrt{|\Delta|}}{4Jd}.\label{eq:10}\end{equation}

\vspace{0.2cm}

\phantomsection\label{case:2}
\noindent\textbf{Case 2:} $q_n\geq \sqrt{|\Delta|}$.\hspace{\parindent}We proceed as in Case \hyperref[case:1]{1}, but with $p_{n-1}/q_{n-1}$ in place of $p_{n+1}/q_{n+1}$. The interval from which we choose $j'$ is centered at $$z_n-\frac{1}{q_n|2q_n\Im(\zeta)/\sqrt{|\Delta|}-p_n|}.$$ With $\mu'=j'q_n+q_{n-1}$ and $\lambda'\in\cO$ nearest to $\mu'\zeta$, the choice of $j'$ gives the penultimate inequality below, and (\ref{eq:10}) and Case \hyperref[case:2]{2}'s assumption give the final inequality: $$|\Im(\mu'\zeta-\lambda')|=\left|j\bigg|q_n\Im(\zeta)-\frac{p_n\sqrt{|\Delta|}}{2}\bigg|-\bigg|q_{n-1}\Im(\zeta)-\frac{p_{n-1}\sqrt{|\Delta|}}{2}\bigg|\right|=$$ $$\left|j-z_n\right|\cdot\bigg|q_n\Im(\zeta)-\frac{p_n\sqrt{|\Delta|}}{2}\bigg|\leq \frac{J}{2}\bigg|q_n\Im(\zeta)-\frac{p_n\sqrt{|\Delta|}}{2}\bigg|+\frac{\sqrt{|\Delta|}}{2q_n}<\frac{\sqrt{|\Delta|}}{8d}+\frac{1}{2}.$$ Since $J\geq 3$ by Lemma \ref{lem:low}, $\sqrt{|\Delta|}/d\leq 1/J\leq 1/3$, extending the chain of inequalities above to give $|\Im(\mu'\zeta-\lambda')|<13/24$. Next, equality of the first and third expressions in Lemma \ref{lem:cfrac} rearranges to $$q_{n-1}=\frac{1}{|2q_n\Im(\zeta)/\sqrt{|\Delta|}-p_n|}-q_nz_n.$$  So, using (\ref{eq:7}) for the final inequality below, $$|\mu'|=\left|\left(j-z_n\right)q_n+\frac{1}{|2q_n\Im(\zeta)/\sqrt{|\Delta|}-p_n|}\right|\leq \frac{J|q_n|}{2}\leq \frac{Jd}{2}.$$ These are replacements for (\ref{eq:8}) and (\ref{eq:9}) from Case \hyperref[case:1]{1}. The remainder of the proof is unchanged.\hfill$\diam$

\vspace{\baselineskip}

Henceforth assume Case \hyperref[case:2]{2} fails, so $q_n<\sqrt{|\Delta|}$. By Dirichlet's box principle, there are $a,b\in\bN$ with $b< \sqrt{2}J$ such that $|2bq_n\Re(\zeta)-a|< \sqrt{2}/J$, and such that the parity of $a$ makes $(a+bp_n\sqrt{\Delta})/2\in\cO$. Discard prior values of $\mu$ and $\lambda$, and set $\mu=bq_n$ and $\lambda=(a+bp_n\sqrt{\Delta})/2$. Scale $a$ and $b$ by a common factor if necessary so that $(b,\lambda)=\cO$. To summarize approximation quality of $\lambda/\mu$: \begin{equation}\mu=bq_n<b\sqrt{|\Delta|}<\sqrt{2}J\sqrt{|\Delta|},\label{eq:11}\end{equation} from (\ref{eq:10}) we get\begin{equation}|\Im(\mu\zeta-\lambda)|<\frac{b\sqrt{|\Delta|}}{4Jd}<\frac{\sqrt{|\Delta|}}{2\sqrt{2}d},\label{eq:12}\end{equation}and $|\Re(\mu\zeta-\lambda)|<1/\sqrt{2}J$ by choice of $b$. Assume $(\lambda,\mu)$ has a non-principal prime divisor since otherwise we are done---the inequalities above show that (\ref{eq:5}) holds.

\vspace{\baselineskip}

\phantomsection\label{case:3}
\noindent\textbf{Case 3:} $|\Re(\mu\zeta-\lambda)|\geq \max(|\Im(\mu\zeta-\lambda)|,\mu/4Jd)$.\hspace{\parindent}This case will be appealed to again later in the proof. In addition to the inequality defining Case \hyperref[case:3]{3}, we will only use $|\Re(\mu\zeta-\lambda)|<1/\sqrt{2}J$ and $\cJ(|\mu|)\leq J$. 

Using $\lambda'=1$ and $\mu'=0$ in Lemma \ref{lem:jacob}, there exists $j$ in any interval of length $\cJ(|\lambda\mu'-\mu\lambda'|)=\cJ(|\mu|)\leq J$ so that $(1+j\lambda,j\mu)$ has only principal prime divisors. Because $|\Re(\mu\zeta-\lambda)|< 1/\sqrt{2}J$, there exists $x$ in any interval of length $1$ so that $$j=\frac{\Re(\mu\zeta-\lambda)}{|\mu\zeta-\lambda|^2}+\frac{x}{\sqrt{2}\Re(\mu\zeta-\lambda)}$$ for such a $j$. Let $x\in[-2/3,1/3]$. According to the inequality defining Case \hyperref[case:3]{3}, if $y=\Im(\mu\zeta-\lambda)/\Re(\mu\zeta-\lambda)$ then $y\in[-1,1]$. We have $$|j\mu|=\frac{\mu}{|\Re(\mu\zeta-\lambda)|}\left|\frac{1}{1+y^2}+\frac{x}{\sqrt{2}}\right|\leq 4Jd\left|\frac{1}{1+y^2}+\frac{x}{\sqrt{2}}\right|,$$ where the final inequality is also assumed by Case \hyperref[case:3]{3}. Next observe that $$\Re(j\mu\zeta-(1+j\lambda)) = j\Re(\mu\zeta-\lambda)-1= \frac{-y^2}{1+y^2}+\frac{x}{\sqrt{2}},$$ and
$$\Im(j\mu\zeta-(1+j\lambda))=j\Im(\mu\zeta-\lambda)=\frac{y}{1+y^2}+\frac{xy}{\sqrt{2}}.$$

The two expressions above verify that $|j\mu\zeta-(1+j\lambda)|$ is less than $1$ on the given domains for $x$ and $y$. (Local maxima occur at $(x,y)=(-2/3,\pm1),(1/3,\pm 1)$.) Now substitute all three expressions above into $$\frac{|j\mu|}{1-|j\mu\zeta-(1+j\lambda)|}$$ to verify that it is at most $14Jd$. (Local maxima occur at $(x,y)=(1/3,\pm1)$.)\hfill$\diam$

\vspace{\baselineskip}

Henceforth we assume Case \hyperref[case:3]{3} fails, so $|\Re(\mu\zeta-\lambda)|< \max(|\Im(\mu\zeta-\lambda)|,\mu/4Jd)$.\hspace{\parindent}Using (\ref{eq:11}) when $\max(|\Im(\mu\zeta-\lambda)|,\mu/4Jd)=\mu/4Jd$ (\ref{eq:12}) and otherwise, we have $|\Re(\mu\zeta-\lambda)|<b\sqrt{|\Delta|}/4Jd$. Thus \begin{equation}|\mu\zeta-\lambda|<\frac{b\sqrt{|\Delta|}}{2\sqrt{2}Jd}<\frac{\sqrt{|\Delta|}}{2d}.\label{eq:13}\end{equation} 

Recall that $\lambda=(a+bp_n\sqrt{|\Delta|})/2$ and $\mu=bq_n$. Since $\text{gcd}(p_n,q_n) =1$ and $(b,\lambda)=\cO$, there is at least one congruence class modulo $2\mu$ whose elements, call one $a'\in\bZ$, make $a'p_n\equiv -a\,\text{mod}\,2q_n$ and $(b,(a'+b\sqrt{\Delta})/2)=\cO$. Indeed, primes dividing $b$ but not $2q_n$ can be avoided by $(a'+b\sqrt{\Delta})/2$ with the Chinese remainder theorem. An odd prime dividing $b$ and $q_n$ cannot divide $|(a'+b\sqrt{\Delta})/2|^2$ lest it divide $a'$, thus $a$, thus $\lambda$, contradicting $(b,\lambda)=\cO$. And regarding $2$, if $b$ is even and $q_n$ is odd, choose $a'\,\text{mod}\,4$ so exactly one of $a'/2$ and $b/2$ is even. If $b$ and $q_n$ are both even then $p_n$ is odd, so exactly one of $a'/2$ and $b/2$ is even automatically because $(b,\lambda)=\cO$ implies the same is true of $a/2$ and $b/2$, and $a'\equiv a\,\text{mod}\,4$. 

If $|\Re(\mu\zeta-\lambda)|<\mu/4Jd$, then choose $a'$ with minimal magnitude from our congruence class $\text{mod}\,2\mu$. So $|a'|\leq \mu$. Otherwise, fix $a'$ nearest to \begin{equation}\frac{b\Re(\lambda-\mu\zeta)\sqrt{|\Delta|}}{\Im(\mu\zeta-\lambda)}\label{eq:14}\end{equation} without exceeding it in magnitude. Then, either from the bound on $\mu$ in (\ref{eq:11}) or the defining inequality of Case \hyperref[case:4]{4}, $|a'|<b\sqrt{|\Delta|}$. Thus with $\mu'=(a'+b\sqrt{|\Delta|})/2$, we have \begin{equation}|\mu'|<\frac{b\sqrt{|\Delta|}}{\sqrt{2}}<J\sqrt{|\Delta|}\leq d.\label{eq:15}\end{equation} Let $\lambda'\in\cO$ be nearest to $\mu'\zeta$.

The congruence restriction on $a'$ makes $\Im(\mu'\lambda/\mu)$, which is $(a+a'p_n)\sqrt{|\Delta|}/4q_n$, a multiple of $\sqrt{|\Delta|}/2$. This is therefore the imaginary part of $\lambda'$ since it minimizes $$|\Im(\mu'\zeta-\lambda')|=\frac{|\Im(\mu'(\mu\zeta-\lambda))|}{\mu}=\frac{|a'\Im(\mu\zeta-\lambda)+b\Re(\mu\zeta-\lambda)\sqrt{|\Delta|}|}{2\mu}.$$ If $|\Re(\mu\zeta-\lambda)|<\mu/4Jd$, then the triangle inequality, $|a'|\leq\mu$, and (\ref{eq:12}) bound the final expression above by $b\sqrt{|\Delta|}/4Jd$. Otherwise, the choice of $a'$ from (\ref{eq:14}) bounds it by $|\Im(\mu\zeta-\lambda)|$. So in both cases, \begin{equation}|\Im(\mu'\zeta-\lambda')|<\frac{b\sqrt{|\Delta|}}{4Jd}<\frac{\sqrt{|\Delta|}}{2\sqrt{2}d}\leq\frac{1}{2\sqrt{2}J}.\label{eq:16}\end{equation} 

We claim that prime divisors of $(\lambda,\mu,\mu')$ are ramified. Suppose $\fp$ is a split prime that divides $(\lambda,\mu)$. Then $\mu\in\bZ$ implies $\overline{\fp}\,|\,\mu$, and thus $\overline{\fp}\nmid\lambda$ because $(\lambda,\mu)$ has no nontrivial rational divisors by choice of $a$ and $b$ (and $\text{gcd}(p_n,q_n)=1$). So $\lambda\mu'-\overline{\lambda\mu'}=2\Im(\lambda\mu')\in(\mu)\subset \overline{\fp}$ shows that $\mu'\in\overline{\fp}$. But $(b,(a'+b\sqrt{\Delta})/2)=\cO$ implies $\mu'$ has no nontrivial rational divisors, giving $\mu'\not\in\fp$.

It is possible that we have just found a new approximation to $\zeta$ with which $\lambda/\mu$ can be combined as in previous cases. On the other hand, we may have just rediscovered $\lambda/\mu$.

\vspace{\baselineskip}

\phantomsection\label{case:4}
\noindent\textbf{Case 4:} $\lambda\mu'\neq\lambda'\mu$.\hspace{\parindent}Since $(\lambda,\mu,\lambda',\mu')$ has only ramified divisors, there is $k\in\bN$ that generates $(\lambda,\mu,\lambda',\mu')^2$. If $k\neq 1$ let $\kappa$ be either $\sqrt{\Delta}/2$ or $(k+\sqrt{\Delta})/2$, whichever belongs to $(\lambda,\mu,\lambda',\mu')$. If $k=1$ let $\kappa = 0$.  Then $$\left(\lambda,\mu,\frac{\kappa\lambda}{k}+\lambda',\frac{\kappa\mu}{k}+\mu'\right)=\cO.$$ To see this, first consider a prime $\fp$ dividing $(\lambda,\mu)$ but not $(\lambda',\mu')$. Our choice of $k$ gives $\fp\nmid (k)$. Thus $\fp\,|\,(\kappa\lambda/k,\kappa\mu/k)$, implying $\fp\nmid(\kappa\lambda/k+\lambda',\kappa\mu/k+\mu')$. Now suppose $\fp\,|\,(\lambda,\mu,\lambda',\mu')$. We have shown $\fp$ is ramified, so $\fp^2$ is generated by a rational prime that can divide neither $(\lambda,\mu)$ nor $(\kappa)$. In particular, $(\lambda,\mu)$ and $(\kappa)$ have order $1$ at $\fp$, while $(k)$ has order $2$ at $\fp$. So $\fp$ does not divide $(\kappa\lambda/k,\kappa\mu/k)$, and so cannot divide $(\kappa\lambda/k+\lambda',\kappa\mu/k+\mu')$. 

Now observe that $\lambda(\kappa\mu/k+\mu')-\mu(\kappa\lambda/k+\lambda')=\lambda\mu'-\mu\lambda'$, which is nonzero by Case \hyperref[case:4]{4}'s assumption. So Lemma \ref{lem:jacob} applies if 2 does not split. 

When $2$ splits, Lemma \ref{lem:jacob}'s hypothesis is still met provided $\mu=bq_n$ is even. Indeed, if $2\,|\,b$, recall that $(b,\mu')=\cO$ to get $(2,\kappa\mu/k+\mu')=\cO$. On the other hand, if $2\nmid b$ and $2\,|\,q_n$, then $p_n$, $a$, and $a'$ must all be odd. Hence $\lambda$ and $\mu'$ are each in one of the primes over $2$. Since $(\lambda,\mu,\mu')$ has no split divisors, $(2,\lambda,\mu')=\cO$, implying $(2,\lambda)$ and $(2,\mu')$ are conjugates of one another. Either $j=1$ or $j=0$ is such that exactly one of $\lambda'$ and $j+\kappa$ is in $(2,\mu')$. With this value of $j$ we see that $(2,j\lambda+(\kappa\lambda/k+\lambda'),j\mu+(\kappa\mu/k+\mu'))=\cO$.

If $2$ splits and $\mu$ is odd, it is possible Lemma \ref{lem:jacob}'s hypothesis is not met. In this case we replace $\mu'$ with $2\mu'$ and again let $\lambda'$ be nearest to $\mu'\zeta$. (So double $\lambda'$ and shift by $-1$, $0$, or $1$.) This doubles the upper bounds in (\ref{eq:15}) and (\ref{eq:16}). Remark that the new values also satisfy $\lambda\mu'\neq\lambda'\mu$ as $2$ divides one side of the non-equality but not the other if $\lambda'$ is replaced with $2\lambda'\pm 1$.

Regardless of whether doubling $\mu'$ is required, \begin{equation}|\mu'|<\frac{Jd}{3}\hspace{0.5cm}\text{and}\hspace{0.5cm}|\Im(\mu'\zeta-\lambda')|<\frac{1}{6\sqrt{2}}.\label{eq:17}\end{equation} This is because when $2$ splits and $\mu$ is odd, Lemma \ref{lem:low} doubles the minimum value of $J$ from $3$ to $6$, cancelling the doubling of (\ref{eq:15}) and (\ref{eq:16}). Note that the rational prime between $2$ and $|\delta|$ required in Lemma \ref{lem:low}'s hypothesis exists because we assume that such a non-inert prime divides $\|(\lambda,\mu)\|=\|(\lambda,q_n)\|$. Thus it divides $q_n$, which is odd and at most $\sqrt{|\Delta|}$.

Let us bound $|\lambda\mu'-\lambda'\mu|$ in view of applying Lemma \ref{lem:jacob}. Combining (\ref{eq:11}), (\ref{eq:13}), (\ref{eq:17}), and $|\Re(\mu'\zeta-\lambda')|\leq 1/2$, we have $$|\lambda\mu'-\lambda'\mu|\leq \mu|\mu'\zeta-\lambda'| + |\mu'(\mu\zeta-\lambda)|<$$ $$\sqrt{2}J\sqrt{|\Delta|}\sqrt{\frac{1}{4}+\frac{1}{72}}+\frac{Jd}{3}\frac{\sqrt{|\Delta|}}{2d}<2J\sqrt{|\Delta|}\leq 2d.$$ Therefore Lemma \ref{lem:jacob} gives $j$ in any interval of length $\cJ(|\lambda\mu'-\mu\lambda'|)\leq J$ such that $$\left(j\lambda+\Big(\frac{\kappa\lambda}{k}+\lambda'\Big),j\mu+\Big(\frac{\kappa\mu}{k}+\mu'\Big)\right)$$ has only principal prime divisors. Choose $j$'s interval to make $|j+\Re(\kappa)/k|\leq J/2$, which we note to be at most $d/2\sqrt{|\Delta|}$. We claim the associated hemisphere completely covers $S_{\zeta}$.

By definition of $k$, it is a product of distinct ramified (rational) primes, which thus divides $\Delta$. Therefore, $|\Im(\kappa)|/k=|\Delta|/2k\sqrt{|\Delta|}\leq |\delta|/2\sqrt{|\Delta|}\leq d/2\sqrt{|\Delta|}$. Alongside our bound on $|\Re(j+\kappa/k)|$, this gives $|j+\kappa/k|\leq d/\!\sqrt{2|\Delta|}$. The remaining expressions below already have bounds from (\ref{eq:11}), (\ref{eq:13}), and (\ref{eq:17}): $$\frac{|j+\kappa/k|\mu+|\mu'|}{1-|(j+\kappa/k)(\mu\zeta-\lambda)|-|\mu'\zeta-\lambda'|}.$$ This is positive and less than $14Jd$.\hfill$\diam$

\vspace{\baselineskip}

Now suppose Case \hyperref[case:4]{4} fails, so $\lambda\mu'=\mu\lambda'$. First, $(b,\lambda)=(b,\mu')=\cO$ forces $b = 1$ since otherwise $b$ divides $\mu\lambda'$ but not $\lambda\mu'$. Inequalities (\ref{eq:11}) and (\ref{eq:13}) become \begin{equation}\mu<\sqrt{|\Delta|}\hspace{0.5cm}\text{and}\hspace{0.5cm}|\mu\zeta-\lambda|<\frac{\sqrt{|\Delta|}}{2\sqrt{2}Jd}.\label{eq:18}\end{equation} And similarly for $\lambda'/\mu'$ from (\ref{eq:15}) and (\ref{eq:16}): \begin{equation}|\mu'|<\frac{\sqrt{|\Delta|}}{\sqrt{2}}\hspace{0.5cm}\text{and}\hspace{0.5cm}|\Im(\mu'\zeta-\lambda')|<\frac{\sqrt{|\Delta|}}{4Jd}.\label{eq:19}\end{equation}(And still $|\Re(\mu'\zeta-\lambda')|\leq 1/2$.) This also implies $\mu=\|(\lambda,\mu)\|$ since a rational divisor of $\mu/\|(\lambda,\mu)\|$ has to divide $\mu'=(a'+b\sqrt{\Delta})/2=(a'+\sqrt{\Delta})/2$ when $\lambda\mu'=\mu\lambda'$.

Pick coprime $\alpha,\beta\in\cO$ for which $\alpha\mu-\beta\lambda$ has minimal nonzero norm in $(\lambda,\mu)$. For $j,j'\in\bZ$, set $\lambda_{j,j'}=\alpha+j\lambda+j'\lambda'$ and $\mu_{j,j'}=\beta+j\mu+j'\mu'$. Observe that $\mu\lambda_{j,j'}-\lambda\mu_{j,j'}=\alpha\mu-\beta\lambda$.

Let $r$ denote the radius of $S_{\zeta}$. So $r\leq 1/14Jd$. Let $\xi\in\bC$ be any point for which $\xi+\lambda/\mu$ lies in the projection of $S_{\zeta}$. Then $t=\sqrt{r^2-|\zeta-(\xi+\lambda/\mu)|^2}$ makes $(\xi+\lambda/\mu,t)\in S_{\zeta}\subset\bH$. We will find $j,j'\in\bZ$ so that $(\lambda_{j,j'},\mu_{j,j'})$ is principal, $|\mu_{j,j'}|< 1/r$, and \begin{equation}t^2<\frac{1}{|\mu_{j,j'}|^2}-\left|\frac{\lambda_{j,j'}}{\mu_{j,j'}}-\left(\xi+\frac{\lambda}{\mu}\right)\right|^2.\label{eq:20}\end{equation} This will complete the proof since (\ref{eq:20}) is equivalent to $S_{\zeta}$ being covered by the hemisphere of $\lambda_{j,j'}/\mu_{j,j'}$ above the arbitrary point $\xi+\lambda/\mu$.

Let $D\subset\bC$ be the open disc that has $0$ on its boundary, radius $$R=\min\!\left(\frac{1}{2r},\,\frac{|\alpha\mu-\beta\lambda|}{2\mu|\xi|}\right),$$ and center with argument matching that of $\overline{\xi}(\alpha\mu-\beta\lambda)$ (or any argument if $\xi=0$). First we show that any $j$ and $j'$ which make $(\lambda_{j,j'},\mu_{j,j'})$ principal and $\mu_{j,j'}\in D$ suffice. Then we show that such $j$ and $j'$ exist.

If $\mu_{j,j'}\in D$ then $|\mu_{j,j'}|<2R$ because $0$ lies on the boundary of $D$. In particular, $|\mu_{j,j'}|<1/r$. So only (\ref{eq:20}), which rearranges to \begin{equation}\left|\frac{\alpha\mu-\beta\lambda}{\mu}-\mu_{j,j'}\xi\right|^2< 1-t^2|\mu_{j,j'}|^2,\label{eq:21}\end{equation} needs to be checked. 

By definition of $r$, $|\mu\zeta-\lambda|/r$ is the magnitude of an element of $(\lambda,\mu)$. So the choice of $\alpha$ and $\beta$ makes $|\alpha\mu-\beta\lambda|\leq|\mu\zeta-\lambda|/r$. If $\xi=0$, this observation combines with $|\mu_{j,j'}|< 1/r$ and $t^2 = r^2-|\zeta-\lambda/\mu|^2$ to give (\ref{eq:21}). 

So suppose $\xi\neq 0$. View $\mu_{j,j'}$ and the center of $D$ as vectors in the plane extending from the origin. The argument of the center of $D$ is equal to that of $\overline{\xi}(\alpha\mu-\beta\lambda)$, so the angle between our two vectors is defined by $$\cos\theta=\frac{\Re\big(\overline{\mu_{j,j'}\xi}(\alpha\mu-\beta\lambda)\big)}{\big|\mu_{j,j'}\xi(\alpha\mu-\beta\lambda)\big|}.$$ Since $0$ is on the boundary of $D$, $\mu_{j,j'}\in D$ is equivalent to $2R\cos\theta>|\mu_{j,j'}|$, or, using the expression for $\cos\theta$ above, $$2R\,\Re\big(\overline{\mu_{j,j'}\xi}(\alpha\mu-\beta\lambda)\big)> |\mu_{j,j'}^2\xi(\alpha\mu-\beta\lambda)|.$$ This, in turn, is equivalent to \begin{equation}\left|\frac{\alpha\mu-\beta\lambda}{\mu}-\mu_{j,j'}\xi\right|^2<\frac{|\alpha\mu-\beta\lambda|^2}{\mu^2}-\left(\frac{|\alpha\mu-\beta\lambda|}{R\mu|\xi|}-1\right)|\mu_{j,j'}\xi|^2,\label{eq:22}\end{equation} which is our first step toward proving (\ref{eq:21}) when $\xi\neq 0$. 

If $R=|\alpha\mu-\beta\lambda|/2\mu|\xi|$, the right-hand side of (\ref{eq:22}) simplifies as shown in the first expression below. Then the first inequality uses $|\alpha\mu-\beta\lambda|\leq \mu$, and the second inequality uses $t< r\leq \max(r,\mu|\xi|/|\alpha\mu-\beta\lambda|)=1/2R = \mu|\xi|/|\alpha\mu-\beta\lambda|$: $$\frac{|\alpha\mu-\beta\lambda|^2}{\mu^2}-|\mu_{j,j'}\xi|^2\leq 1-\left|\frac{\mu\mu_{j,j'}\xi}{\alpha\mu-\beta\lambda}\right|^2< 1-t^2|\mu_{j,j'}|^2.$$ Connecting with (\ref{eq:22}) completes a chain of inequalities that gives (\ref{eq:21}). 

If $R=1/2r$ then $|\xi|\leq r|\alpha\mu-\beta\lambda|/\mu$ by definition of $R$. Furthermore, we have seen that $r|\alpha\mu-\beta\lambda|/\mu\leq |\zeta-\lambda/\mu|$ by choice of $\alpha$ and $\beta$. From these we get the last inequality below: $$t^2=r^2-\left|\zeta-\left(\xi+\frac{\lambda}{\mu}\right)\right|^{\,2}\leq r^2-\left(\left|\zeta-\frac{\lambda}{\mu}\right|-|\xi|\right)^{\!2}\leq r^2-\left(\frac{r|\alpha\mu-\beta\lambda|}{\mu}-|\xi|\right)^{\!2}.$$ Remember that we aim to bound the right-hand side of (\ref{eq:22}) by $1-t^2|\mu_{j,j'}|^2$ (thereby proving (\ref{eq:21})). Replace $t^2$ with its bound above, and the desired inequality becomes $$\frac{|\alpha\mu-\beta\lambda|^2}{\mu^2}-\left(\frac{2r|\alpha\mu-\beta\lambda|}{\mu|\xi|}-1\right)|\mu_{j,j'}\xi|^2\leq$$ $$1-\left(r^2-\left(\frac{r|\alpha\mu-\beta\lambda|}{\mu}-|\xi|\right)^{\!2}\right)|\mu_{j,j'}|^2.$$ This rearranges to $$(1-r^2|\mu_{j,j'}|^2)\left(1-\frac{|\alpha\mu-\beta\lambda|^2}{\mu^2}\right)\geq 0,$$ which is true.

Now we turn to the existence of such a $j$ and $j'$.

\vspace{\baselineskip}
\phantomsection\label{case:5}
\noindent\textbf{Case 5:} $|\zeta-\lambda/\mu|\leq1/2\sqrt{2}Jd$.\hspace{\parindent}Since $(\alpha,\beta)=\cO$, Lemma \ref{lem:jacob} gives $j'$ in any interval of length $\cJ(2\sqrt{\|(\lambda,\mu)\|})=\cJ(2\sqrt{\mu})\leq J$ such that $(\alpha+j'\lambda',\beta+j'\mu')$ and $2(\lambda,\mu)$ share only principal prime divisors. Take such a $j'$ that makes the distance between $\Im(\mu_{j,j'})$ (which does not depend on $j$) and the imaginary part of the center of $D$ at most $J|\Im(\mu')|/2 = J\sqrt{|\Delta|}/4$. 

Then Lemma \ref{lem:jacob} applies again: any interval of length $\cJ(|\alpha\mu-\beta\lambda|)\leq\cJ(|\mu|)\leq J$ contains $j\in\bZ$ making $(\lambda_{j,j'},\mu_{j,j'})$ principal. Take such a $j$ that makes the distance between $\Re(\mu_{j,j'})$ and the real part of the center of $D$ at most $J\mu/2 < J\sqrt{|\Delta|}/2$.

Let us check that $\mu_{j,j'}\in D$. If $\alpha\mu-\beta\lambda$ is not $\pm\mu$, then it is not rational. In particular, $$\frac{2\mu|\xi|}{|\alpha\mu-\beta\lambda|}\leq \frac{4\mu|\xi|}{\sqrt{|\Delta|}}<4|\xi|\leq 4\left(\left|\zeta-\frac{\lambda}{\mu}\right|+r\right)<4\left(\frac{1}{2\sqrt{2}Jd}+\frac{1}{14Jd}\right)<\frac{2}{Jd}.$$ Since $2\mu|\xi|/|\alpha\mu-\beta\lambda|<4|\xi|$ (the first two inequalities above) also holds when $\alpha\mu-\beta\lambda=\pm\mu$, either way we get $R=\min(1/2r,|\alpha\mu-\beta\lambda|/2\mu|\xi|)> Jd/2.$ But the distance from $\mu_{j,j'}$ to the center of $D$ is at most $$\frac{J\sqrt{\mu^2+\Im(\mu')^2}}{2}<\frac{J\sqrt{5|\Delta|}}{4}<\frac{J^2\sqrt{|\Delta|}}{2}\leq \frac{Jd}{2},$$ giving $\mu_{j,j'}\in D$.\hfill$\diam$

\vspace{\baselineskip}

Suppose Case \hyperref[case:5]{5} fails, so $|\zeta-\lambda/\mu|>1/2\sqrt{2}Jd$. Then $|\Re(\mu\zeta-\lambda)|<|\Im(\mu\zeta-\lambda)|$, because otherwise $|\Re(\mu\zeta-\lambda)|\geq|\mu\zeta-\lambda|/\sqrt{2}>\mu/4Jd$ violates the failure of Case \hyperref[case:3]{3}. In other words, it is the argument of $\mu\zeta-\lambda$ that prohibits Case \hyperref[case:3]{3}. We claim if $(\lambda,\mu,\mu')\neq \cO$ or $(\lambda,\mu)$ is not reduced, then $\mu\zeta-\lambda$ can be rotated (and perhaps scaled slightly in the process) to allow for Case \hyperref[case:3]{3}.

First suppose $(\lambda,\mu,\mu')$ has at least one prime divisor, which we know to be ramified. So $(\lambda,\mu,\mu')^2$ is generated by some rational integer $k>1$. If $\kappa=\sqrt{|\Delta|}/2$ belongs to $(\lambda,\mu,\mu')$, then $(\kappa\lambda/k,\kappa\mu/k)\subset\cO$ and $|\Re(\kappa(\mu\zeta-\lambda))|\geq|\Im(\kappa(\mu\zeta-\lambda))|$ since $\text{arg}(\kappa/k)=\pi/2$. Otherwise $\kappa=(k+\sqrt{|\Delta|})/2$ belongs to $(\lambda,\mu,\mu')$, as does $\overline{\kappa}$. It follows from $0<k<\mu<\sqrt{|\Delta|}$ that $\pi/4<\text{arg}\,\kappa<\pi/2<\text{arg}\,\overline{\kappa}<3\pi/4$. So by conjugating $\kappa$ if necessary, we may assume $|\Re(\kappa(\mu\zeta-\lambda)/k)|\geq|\Im(\kappa(\mu\zeta-\lambda)/k)|$. 

Let us check that $\kappa\lambda/k$ and $\kappa\mu/k$ qualify for Case \hyperref[case:3]{3}. Using $$\frac{|\kappa|}{k}\leq\frac{\sqrt{|\Delta|}}{\sqrt{2}k}\leq\frac{|\delta|}{\sqrt{2|\Delta|}}\leq\frac{d}{\sqrt{2|\Delta|}}$$ and (\ref{eq:18}) gives $|\kappa\mu/k|< d/\sqrt{2}$. So $\cJ(|\kappa\mu/k|)\leq J$ as Case \hyperref[case:3]{3} requires. Also, $$\left|\Re\left(\frac{\kappa\mu}{k}\zeta-\frac{\kappa\lambda}{k}\right)\right|\geq\frac{1}{\sqrt{2}}\left|\frac{\kappa\mu}{k}\zeta-\frac{\kappa\lambda}{k}\right|=\frac{|\kappa\mu|}{k\sqrt{2}}\left|\zeta-\frac{\lambda}{\mu}\right|>\frac{|\kappa\mu/k|}{4Jd}$$ from the failure of Case \hyperref[case:5]{5}, and $$\left|\Re\left(\frac{\kappa\mu}{k}\zeta-\frac{\kappa\lambda}{k}\right)\right|\leq\frac{d|\mu\zeta-\lambda|}{\sqrt{2|\Delta|}}\leq\frac{1}{4J}<\frac{1}{\sqrt{2}J}$$ from (\ref{eq:18}). This verifies all of Case \hyperref[case:3]{3}'s hypotheses.

Now suppose $(\lambda,\mu)$ is not reduced. Rephrase $|\Re(\mu\zeta-\lambda)|<|\Im(\mu\zeta-\lambda)|$ as $\text{arg}(\mu\zeta-\lambda)\,\text{mod}\,\pi\in(\pi/4,3\pi/4)$. If needed, add a multiple of $\mu$ to $\mu'$ so $\Re(\mu')\leq\mu/2=\|(\lambda,\mu)\|/2$. Then $(\mu'/\mu)(\lambda,\mu)=(\lambda',\mu')$ has minimal norm in the ideal class of $(\lambda,\mu)$, and in particular $|\mu'|<\mu$. Without loss of generality suppose $0\leq a'=2\Re(\mu')$. Since $a'\leq\mu<\sqrt{|\Delta|}/2$ by (\ref{eq:18}), the argument of $\mu'/\mu$ is $\arctan(\sqrt{|\Delta|}/a')\in(\pi/4,\pi/2]$. Therefore $|\Re(\mu'\zeta-\lambda')|\geq|\Im(\mu'\zeta-\lambda')|$
if $$\text{arg}(\mu\zeta-\lambda)\,\text{mod}\,\pi\geq\frac{3\pi}{4}-\arctan\left(\frac{\sqrt{|\Delta|}}{a'}\right).$$ If this does not account for our value of $\mu\zeta-\lambda$, replace $\lambda'$ and $\mu'$ with $\lambda'-\lambda$ and $\mu'-\mu$. The argument of $\mu'/\mu$ is now $\pi-\arctan(\sqrt{|\Delta|}/(2\mu-a'))\in(\pi/2,\pi)$, so $|\Re(\mu'\zeta-\lambda')|\geq|\Im(\mu'\zeta-\lambda')|$ whenever $$\text{arg}(\mu\zeta-\lambda)\,\text{mod}\,\pi\leq\frac{\pi}{4}+\arctan\left(\frac{\sqrt{|\Delta|}}{2\mu-a'}\right).$$ This covers all possible arguments of $\mu\zeta-\lambda$ because $$\frac{3\pi}{4}-\arctan\left(\frac{\sqrt{|\Delta|}}{a'}\right)\leq\frac{\pi}{4}+\arctan\left(\frac{\sqrt{|\Delta|}}{2\mu-a'}\right)$$ is equivalent to $a'/\sqrt{|\Delta|}\leq \sqrt{|\Delta|}/(2\mu-a')$, or $a'(2\mu-a')\leq |\Delta|$. And $a'(2\mu-a')$ is maximized when $a'=\mu$, which is less than $\sqrt{|\Delta|}$. 

Let us check that $\lambda'$ and $\mu'$ qualify for Case \hyperref[case:3]{3}. In replacing $\mu'$ with $\mu'-\mu$, it may now be that $|\mu'|>\mu$. But $|\Re(\mu')|=\mu-a'/2\leq\mu$ means $|\mu'|/\mu\leq\sqrt{1+|\Delta|/4\mu^2}$, which has a maximum of $\sqrt{2}$ for $\mu\geq\sqrt{|\Delta|}/2$. So $|\mu'|<\sqrt{2|\Delta|}$, which means $\cJ(|\mu'|)\leq J$. Furthermore, $$|\Re(\mu'\zeta-\lambda')|\geq\frac{|\mu'\zeta-\lambda'|}{\sqrt{2}}=\frac{|\mu'|}{\sqrt{2}}\left|\zeta-\frac{\lambda}{\mu}\right|>\frac{|\mu'|}{4Jd}$$ from the failure of Case \hyperref[case:5]{5}, and $$|\Re(\mu'\zeta-\lambda')|\leq |\mu'\zeta-\lambda'|\leq\sqrt{2}|\mu\zeta-\lambda|<\frac{\sqrt{|\Delta|}}{2Jd} < \frac{1}{\sqrt{2}J}$$ from $|\mu'|\leq\sqrt{2}\mu$ and (\ref{eq:18}). Thus Case \hyperref[case:3]{3} applies again.

We may therefore assume $(\lambda,\mu)$ is reduced and $(\lambda,\mu,\mu')=\cO$. Assuming $(\lambda,\mu)$ is reduced allows us to use $\alpha=1$ and $\beta=0$ in defining $\lambda_{j,j'}$ and $\mu_{j,j'}$. Assuming $(\lambda,\mu,\mu')=\cO$ gives $(\lambda,\mu,1+j'\lambda',j'\mu')=\cO$ for any $j'$. Also, with the aim of applying Lemma \ref{lem:jacob} to find $j$, observe when $2$ splits that if $(2,1+j'\lambda',j'\mu')\neq\cO$ then $(1+j'\lambda')\mu-(j'\mu')\lambda=\mu$ must be even. Thus $(2,1+\lambda+j'\lambda',\mu+j'\mu')=\cO$ because $\lambda$ and $\mu'$ are in different primes over $2$. So regardless of the choice of $j'$, Lemma \ref{lem:jacob} allows us to choose $j$ in any interval of length $J$ such that $(\lambda_{j,j'},\mu_{j,j'})$ has only split prime divisors. We therefore choose whichever $j'$ makes $\Im(\mu_{j,j'})$ closest to (so within $|\Im(\mu')/2|=\sqrt{|\Delta|}/4$ of) the imaginary part of the center of $D$.

Now, use $|\mu'|\geq|\Im(\mu')|=\sqrt{|\Delta|}/2$ to get the second inequality below, and use (\ref{eq:19}) and $|\Re(\mu'\zeta-\lambda')|\leq 1/2$ for the last inequality: $$\frac{2\mu|\xi|}{|\alpha\mu-\beta\lambda|}=2|\xi|\leq 2\left|\zeta-\frac{\lambda'}{\mu'}\right|+2r\leq \frac{4|\mu'\zeta-\lambda'|}{\sqrt{|\Delta|}}+2r<\frac{3}{\sqrt{|\Delta|}}.$$ Therefore $R> \sqrt{|\Delta|}/3$. Since $\Im(\mu_{j,j'})$ is within $\sqrt{|\Delta|}/4$ of the imaginary part of the center of $D$, the line with fixed imaginary part $\Im(\mu_{j,j'})$ intersects $D$ to make an interval longer than $2R\sqrt{1-(3/4)^2}$. Thus the number of consecutive $j$-values for which $\mu_{j,j'}\in D$ is at least $$\left\lfloor\frac{2R\sqrt{1-(3/4)^2}}{\mu}\right\rfloor\geq\left\lfloor\frac{\sqrt{7}/4}{|\mu\zeta-\lambda| + \mu r}\right\rfloor\geq \left\lfloor \frac{\sqrt{7}J^2/4}{1/2\sqrt{2}+1/14}\right\rfloor> J.$$ Thus there exists $\mu_{j,j'}\in D$ making $(\lambda_{j,j'},\mu_{j,j'})$ principal.\end{proof}

\begin{figure}[ht]
    \hspace{0.185cm}\begin{overpic}[trim=0cm 7.3cm 0cm 7.6cm, clip, scale=0.57,unit=1mm]{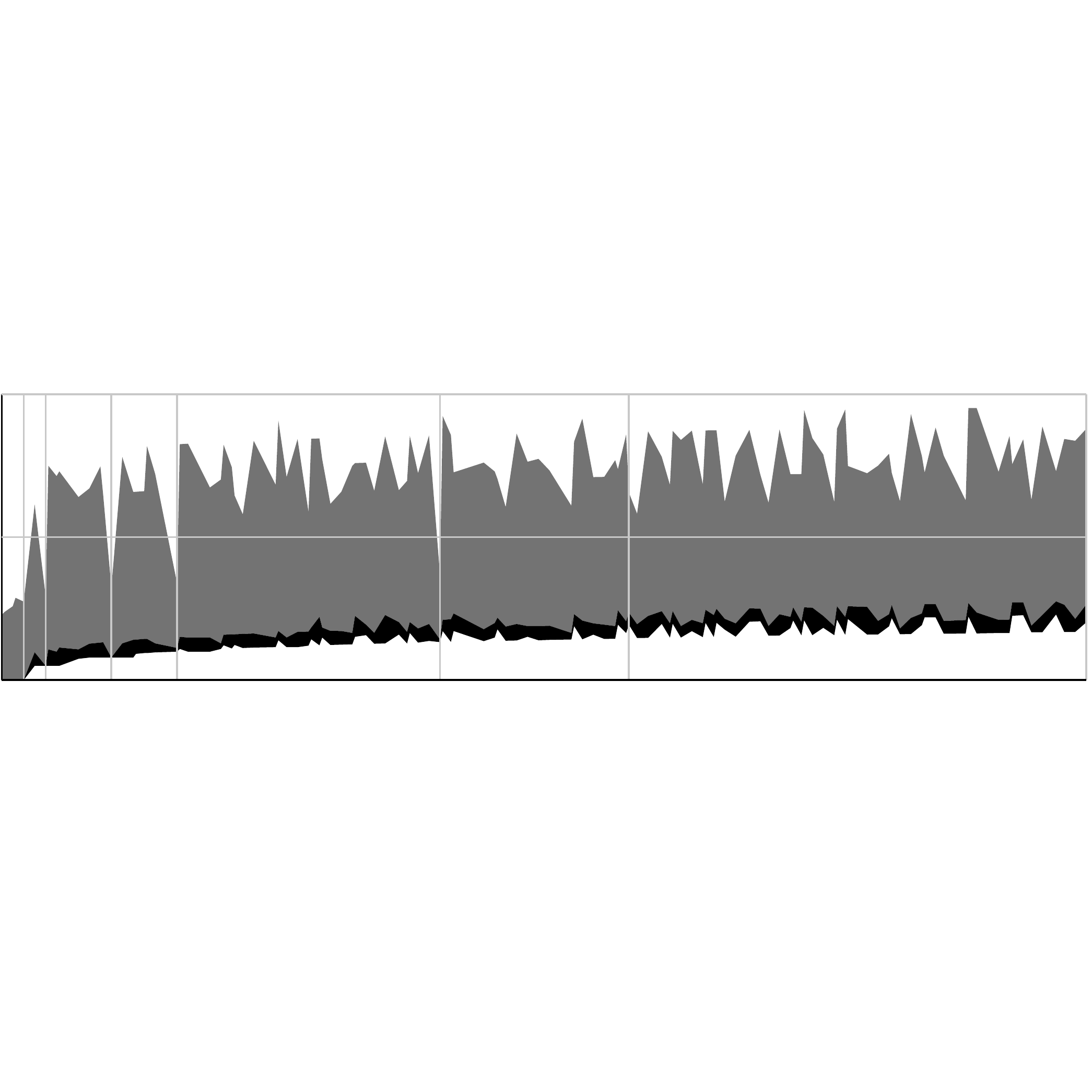}
    \put(-2.5,3.1){$0$}
    \put(-2.5,18.9){$7$}
    \put(-3.9,34.4){$14$}
    \put(0.4,0.8){$1\hspace{-0.03cm}1$}
    \put(4.13,0.8){$19$}
    \put(10.5,0.8){$43$}
    \put(18.2,0.8){$67$}
    \put(46.1,0.8){$163$}
    \put(67.15,0.8){$232$}
    \put(117.5,0.8){$399$}
    \end{overpic}\captionsetup{width=0.931\textwidth}\caption{Logarithms of the upper and lower bound and $\cS(\Gamma)$ for $|\Delta|<400$.}\label{fig:6}
\end{figure}
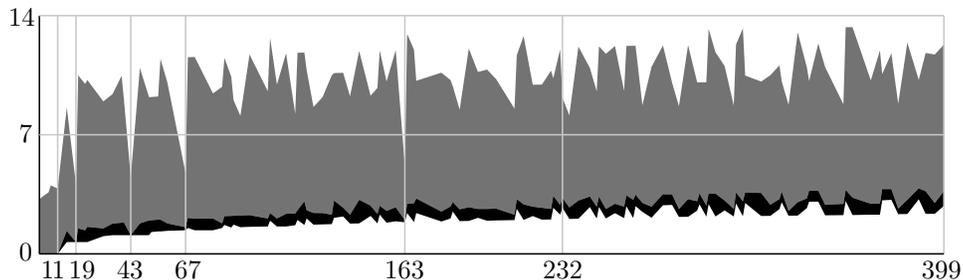

Figure \ref{fig:6} compares the logarithms of our lower bound and upper bound as summarized in Theorem \ref{thm:intro} to $\log(\cS(\Gamma))$, which is the jagged line splitting the gray overapproximation area and the black underapproximation area. The lower bound is rounded up to the nearest integer magnitude. This creates little observable effect except to make its logarithm in the five Euclidean cases ($|\Delta|\leq 11$) nonnegative. Values of $\cS(\Gamma)$ were initially computed using Rahm's Bianchi.gp \cite{bianchigp}, then verified with John Cremona's implementation of Swan's algorithm \cite{cremona}, which also filled in gaps ($\Delta=-119,-183,-267$) due to insufficient RAM for Bianchi.gp.

For $|\Delta|<400$, Swan's number is at most four times our lower bound. Our upper bound appears overly-sensitive to $J$ as defined in Theorem \ref{thm:up}. Indeed, $J$ accounts for all of the fluctuation in the top line except at $|\Delta|=8$ and $232$, the only two cases where $14J\max(J\sqrt{|\Delta|},|\delta|)= 14J|\delta|$. The large dips in the top line that are marked on the graph occur at principal ideal domains since $J=1$ in these rings: $|\Delta|\leq 11$ and $|\Delta|=19,43,67,163$.

\begin{corollary}\label{cor:asymptotic}$\cS(\Gamma)\ll J'\max(|\delta|,J'\!\sqrt{|\Delta|})$, where $J'=(\log |\Delta|/\log\log|\Delta|)^{ 4.27}$.\end{corollary}

\begin{proof}Using Proposition \ref{prop:asymptotic}'s notation, Robin proves $\omega(\fa)<1.4\log\|\fa\|/\log\log\|\fa\|$ when $\|\fa\|\geq 3$ \cite{robin}. This gives $\cJ(x)<c\hspace{0.03cm}(\log x/\log\log x)^{4.27}$ for some constant $c$. So let $J=c\hspace{0.03cm}(\log |\Delta|/\log\log|\Delta|)^{ 4.27}$. Then if $|\Delta|$ is sufficiently large we have $\cJ(2\max(|\delta|,J\sqrt{|\Delta|}))\leq \cJ(|\Delta|) < J$ as required by Theorem \ref{thm:up}.\end{proof}

The proof of Theorem \ref{thm:up} simplifies when applied to the \textit{extended Bianchi group}, which we denote $\widehat{\Gamma}$. (See Section 7.4 of \cite{elstrodt} for a description and basic properties.) The Bianchi group, $\Gamma$, is a normal subgroup of $\widehat{\Gamma}$ with index equal to the twice the number of $2$-torsion ideal classes when $\Delta\neq -3,-4$. In particular, if a presentation for $\widehat{\Gamma}$ is computed using Poincar\'{e}'s polyhedron theorem, then a presentation for $\Gamma$ follows from the Nielsen-Schreier algorithm \cite{nielsen}. 

When bounding Swan's number for $\widehat{\Gamma}$ (still meaning maximum curvature in the Bianchi polyhedron), ramified primes are no longer an obstacle. Up to scaling, $\lambda$ and $\mu$ compose a column or row in $\widehat{\Gamma}$ if and only if primes dividing $(\lambda,\mu)$ are ramified. So there is no need to consider ``$\delta$"---our upper bound takes the form $\cS(\widehat{\Gamma})<cJ^2\sqrt{|\Delta|}$. There is also no need to sieve all non-principal primes with $\cJ$ in Notation \ref{not:J}. Only split primes must be avoided. 

The bound on $\cS(\widehat{\Gamma})$ can be further improved (and the proof further simplified) with the use of \textit{complex} continued fractions in Cases \hyperref[case:1]{1} and \hyperref[case:2]{2}. The author generalized continued fractions to all imaginary quadratic fields in \cite{martin}. The convergents $\lambda_n/\mu_n\in K$ are such that $(\lambda_n,\mu_n)$ and $(\lambda_{n+1},\mu_{n+1})$ may not be coprime when $|\Delta| > 11$ (non-Euclidean rings). Thus $j\lambda_n+\lambda_{n+1}$ and $j\mu_n+\mu_{n+1}$ may not compose a column or row in $\Gamma$ for any $j$, hence the necessity for floor function continued fractions and Lemma \ref{lem:thue} in the proof of Theorem \ref{thm:up}. But $(\lambda_n,\mu_n,\lambda_{n+1},\mu_{n+1})$ is purely ramified, so such linear combinations will work for $\widehat{\Gamma}$. This eliminates the need for Lemma \ref{lem:thue}, in turn eliminating the implicit definition of $J$, giving instead something like $J=\cJ(c'\sqrt{|\Delta|})$ (where, again, $\cJ$ sieves only split primes). Nevertheless, we study only $\Gamma$ here as $\cS(\Gamma)$ is more relevant to literature than $\cS(\widehat{\Gamma})$.

\bibliographystyle{plain}
\bibliography{refs}

\end{document}